\documentclass[a4paper,reqno]{amsart}
\title[Cunningham Chains]{Hopf--Galois structures on separable field extensions of degree related to Cunningham chains}
\author{Andrew Darlington}
\date{\today}

\usepackage[utf8]{inputenc}
\usepackage{amsmath}
\usepackage{amsfonts}
\usepackage{amssymb}
\usepackage{pifont} 
\usepackage{enumerate} 
\usepackage{enumitem}
\usepackage{wasysym} 
\usepackage{mathrsfs}
\usepackage{xfrac} 
\usepackage{array}
\usepackage{siunitx}
\usepackage{mathtools}
\usepackage{amsthm}
\usepackage{framed}
\usepackage{pifont}
\usepackage[english]{babel}
\usepackage{tikz,tkz-euclide}
\usepackage{xcolor,graphicx}
\usepackage{faktor}
\usepackage{float}
\usepackage{hyperref}
\usepackage[figuresright]{rotating}
\usepackage{multirow}

\newtheorem{theorem}{Theorem}[section]
\newtheorem{proposition}[theorem]{Proposition}
\newtheorem{lemma}[theorem]{Lemma}
\newtheorem{conjecture}[theorem]{Conjecture}
\newtheorem{corollary}[theorem]{Corollary}
\newtheorem{definition}[theorem]{Definition}

\theoremstyle{definition}
\newtheorem{remark}[theorem]{Remark}
\newtheorem{example}[theorem]{Example}
\newtheorem{examples}[theorem]{Examples}

\newcommand{\e}{\textbf{e}}
\newcommand{\bt}{\textbf{t}}
\newcommand{\Gal}{\mathrm{Gal}}
\newcommand{\Aut}{\mathrm{Aut}}
\newcommand{\Hol}{\mathrm{Hol}}

\setlist[enumerate]{itemsep=0mm}

\makeatletter
\def\bign#1{\mathclose{\hbox{$\left#1\vbox to8.5\p@{}\right.\n@space$}}\mathopen{}}
\def\Bign#1{\mathclose{\hbox{$\left#1\vbox to11.5\p@{}\right.\n@space$}}\mathopen{}}

\begin{document}
\begin{abstract}
    The past few years have seen Hopf--Galois structures on extensions of squarefree degree studied in various contexts. The Galois case was fully explored by Alabdali and Byott in 2020, followed by a first attempt at generalising these results to include non-normal extensions by Byott and Martin-Lyons; their work looks at separable extensions of degree $pq$ with $p,q$ distinct odd primes, and $p=2q+1$. This paper extends the latter work further by considering separable extensions of squarefree degree $n=p_1...p_m$ where each pair of consecutive primes $p_i,p_{i+1}$ are related by $p_i=2p_{i+1}+1$.
\end{abstract}
\maketitle
\bibliographystyle{amsalpha}
	
\section{Introduction}
Hopf--Galois structures arose in a paper by Chase and Sweedler in 1969 \cite{CS69}. Building on the notion of a Galois extension of commutative rings in \cite{CHR65}, they stated what it meant for a commutative ring $S$ to be an $H$-Galois extension of a commutative ring $R$ for some $R$-Hopf algebra $H$. Almost twenty years later, Greither and Pareigis \cite{GP87} showed that, in the case that $L/K$ is a finite separable field extension, the problem of finding Hopf--Galois structures on $L/K$ can be tackled using a group-theoretic approach. One may view such a structure as essentially what arises when one replaces the Galois group of an extension with some $K$-Hopf algebra $H$, requiring it to fulfil a somewhat similar role - that is $H$ should act on $L/K$ in some `nice' way. 

Let $E$ be the Galois closure of $L/K$, with $G=\Gal(E/K)$, $G'=\Gal(E/L)$, and $X=G/G'$. Then the Greither-Pareigis approach gives a bijective correspondence between Hopf algebras giving a Hopf--Galois structure on $L/K$ and regular subgroups $N$ of $\text{Perm}(X)$ normalised by the image of the left-translation map $\lambda: G \rightarrow \text{Perm}(X)$ given by $\lambda(\sigma)(\overline{\tau})=\overline{\sigma\tau}$ (where the bar represents elements in the coset space $X$). 

The isomorphism class of $N$ is known as the \emph{type} of the Hopf--Galois structure. In particular, we have $H=E[N]^G$, where $G$ acts on $E$ and $N$ by automorphisms and by conjugation (via $\lambda$) respectively, and $H$ acts on $L/K$ as follows:
\[\left(\sum_{\eta \in N}a_{\eta}\eta\right)\cdot x = \sum_{\eta \in N}a_{\eta}(\eta^{-1}(1_X))(x).\]
\begin{examples}
    Let all notation be as above, but with $L/K$ now Galois.
    \begin{enumerate}
        \item Consider the right translation map $\rho: G \rightarrow \text{Perm}(G)$ given by $\rho(\sigma)(\tau)=\tau\sigma^{-1}$. Then $H=L[\rho(G)]^G=K[\rho(G)]$ gives the canonical classical Hopf--Galois structure on $L/K$.
        \item Now consider the left translation map $\lambda: G \rightarrow \text{Perm}(G)$. Then $H=L[\lambda(G)]^G$ gives the canonical non-classical Hopf--Galois structure on $L/K$.
    \end{enumerate}
    Note that if $G$ is abelian, then $\rho(G)=\lambda(G)$, and so $L[\lambda(G)]^G=L[\rho(G)]^G$.
\end{examples}
Aside from these well-known examples, it is possible for many other Hopf algebras to act on an extension in the way described above. It is even possible for the same Hopf algebra to act on $L/K$ in two or more ways such that the resulting Hopf--Galois structures are distinct, even in the case that $L/K$ is Galois. This then, for example, greatly enlarges the study of Galois module theory as each Hopf--Galois structure essentially offers a different lens through which to view the extension. Additionally, there exists a `Hopf--Galois correspondence' between the Hopf subalgebras of $H$ and the intermediate extensions of $L/K$. Unlike the classical Galois correspondence, although it is always injective, it is not so often surjective. The growing interest in Hopf--Galois theory has led to the classification of Hopf--Galois structures in many different cases; typically with restrictions on the degree of the field extension, as in \cite{CS20}, \cite{CS21} and \cite{Byo04} (the latter in the case of Galois extensions), or specifying Hopf--Galois structures of certain types, as in \cite{Kohl98} and \cite{AJ23}. In 2019, Byott and Alabdali \cite{AB20} gave a count and classification of the number of Hopf--Galois structures on Galois extensions of squarefree degree. In a first attempt at generalising this work to separable (but not necessarily normal) extensions, Byott and Martin-Lyons in \cite{BML21} looked at Hopf--Galois structures on extensions of degree $pq$, where $p=2q+1$ with $(p,q)$ a safe prime - Sophie Germain prime pair.

In a further generalisation, a previous work of the author's \cite{Dar24(a)} extended this to the situation where $p,q$ are arbitrary distinct odd primes. Initial further work suggested that the problem of finding Hopf--Galois structures on separable squarefree degree $n$ extensions becomes significantly more involved each time another prime is added to the factorisation of $n$, and so it is of interest to find special cases or trends which could give an insight into the more general case. In this paper, therefore, we explore the case where $n$ can be written as a product of distinct odd primes $p_1>\cdots>p_l$, such that each consecutive prime pair $(p_i,p_{i+1})$ is related by $p_i=2p_{i+1}+1$, thus the primes $p_2,\cdots,p_l$ are Sophie-Germain primes, and the primes $p_1,\cdots,p_{l-1}$ are safe primes. The $l$-tuple $(p_1,\cdots,p_l)$ is called a \emph{Cunningham chain} of length $l$. In 1904, Dickson formulated the following conjecture which extended Dirichlet's theorem on arithmetic progressions:
\begin{conjecture}[\cite{Dic04}, this formulation due to Ribenboim in Chapter 6 of \cite{Rib96}]\label{Dickson}
    Let $k\geq 1$ be an integer, $a_1,\cdots,a_k$ a sequence of integers and $b_1,\cdots,b_k$ be a sequence of integers with $b_i \geq 1$ for each $1 \leq i \leq k$. Further, suppose that there does not exist an integer $m>1$ such that $m$ divides the product
    \[\prod_{i=1}^k(a_i+nb_i)\]
    for all positive integers $n$. Then there are infinitely many positive integers $n$ such that the linear forms
    \[a_1+b_1n, \; a_2+b_2n, \; \cdots \;, a_k+b_kn\]
    are all prime.
\end{conjecture}
One notes that the case $k=1$ corresponds to Dirichlet's Theorem. We also get the following two notable corollaries:
\begin{conjecture}[Twin Prime Conjecture]\label{twin_prime}
    There are infinitely many pairs of primes of the form $n,n+2$.
\end{conjecture}
\begin{conjecture}\label{cunningham}
    For any $l$, there are infinitely many Cunningham chains of length $l$.
\end{conjecture}
\begin{proposition}
    Conjectures \ref{twin_prime} and \ref{cunningham} are both corollaries of Conjecture \ref{Dickson}.
\end{proposition}
\begin{proof}
    For the twin prime conjecture, we take $k=2$ with $a_1=0, a_2=2$, and $b_1=b_2=1$. For Conjecture \ref{cunningham}, we take $k=l$, $a_i=2^{i-1}$ and $b_i=2^{i-1}-1$.
\end{proof}
We note that taking $l=2$ gives the conjecture for Sophie-Germain primes. We now give some preliminary results and definitions.

\section{Preliminaries}
\begin{definition}
    A positive integer $n$ will be called a \emph{Cunningham product} if it can be expressed as a product of primes which form a Cunningham chain of some length.
\end{definition}
    Let $L/K$ be a separable field extension of degree $n$ with $E,G,G'$ and $X$ as above.
\begin{definition}
    For an abstract group $N$ of order $n$, define the holomorph, $\Hol(N)$, of $N$ as
    \[\Hol(N) = N \rtimes \Aut(N).\]
\end{definition}
For $\eta,\mu \in N$ and $\alpha,\beta \in \Aut(N)$, we have
\[[\eta,\alpha][\mu,\beta]=[\eta\alpha(\mu),\alpha\beta].\]
As is convention, we will write $\eta$ for $[\eta,\text{id}_N]$ and $\alpha$ for $[1_N,\alpha]$.
\begin{theorem}[Byott's Translation, \cite{Chi00}]
    There is a bijection between the following two sets of homomorphisms:
    \[\mathcal{N}:=\{\varphi:N \rightarrow \mathrm{Perm}(X) \mid \varphi(N) \text{ regular}\}\]
    and
    \[\mathcal{G}:=\{\vartheta:G\rightarrow \mathrm{Perm}(N) \mid \vartheta(G')=\mathrm{Stab}(1_N)\}.\]
    Furthermore, if $\varphi,\varphi' \in \mathcal{N}$ correspond to $\vartheta,\vartheta' \in \mathcal{G}$, then $\varphi(N)=\varphi'(N)$ iff $\vartheta(G)$ and $\vartheta'(G)$ are conjugate and element of $\Aut(N)$; and $\varphi(N)$ is normalised by $\lambda(G)$ iff $\vartheta(G)$ is contained in $\Hol(N)$.
\end{theorem}
This leads to the following counting formula:
\begin{lemma}[\cite{Byo96}]\label{Byott_num_HGS}
    Let $G,G',N$ be as above, let $e(G,N)$ be the number of Hopf--Galois structures of type $N$ which realise $G$, and $e'(G,N)$ the number of subgroups $M$ of $\emph{Hol}(N)$ which are transitive on $N$ and isomorphic to $G$ via an isomorphism taking the stabiliser $M'$ of $1_N$ in $M$ to $G'$. Then
    \[e(G,N)=\frac{|\emph{Aut}(G,G')|}{|\emph{Aut}(N)|}e'(G,N),\]
    where
    \[\Aut(G,G')=\left\{\theta\in\Aut(G) \mid \theta(G')=G'\right\},\]
    the group of automorphisms $\theta$ of $G$ such that $\theta$ fixes the identity coset $1_GG'$ of $X=G/G'$.
\end{lemma}
In this paper, therefore, in order to find and count Hopf--Galois structures on separable extensions whose degree $n$ is a Cunningham product, we compute the transitive subgroups of $\Hol(N)$ for each group $N$ of order $n$. The following proposition aids the counting of these:
\begin{lemma}
    Let $l$ be a positive integer and $1 \leq k \leq l$. Then there are
    \begin{equation}\label{k_c_2_subgroups}
        \prod_{i=1}^k\frac{2^{l-(i-1)}-1}{2^i-1}
    \end{equation}
    subgroups of $(C_2)^l:=\underset{l -\text{times}}{C_2 \times \cdots \times C_2}$ of order $2^k$.
\end{lemma}
\begin{proof}
    We note that $(C_2)^l$ can be identified with the vector space $\mathbb{F}_2^l$, and that a subgroup of $(C_2)^l$ of rank $k$ corresponds to a $k$-dimensional subspace of $\mathbb{F}_2^l$.

    Now, a $k$-dimensional subspace of $\mathbb{F}_2^l$ is generated by $k$-tuples $(v_1,\ldots,v_k)$ of linearly independent vectors. We claim that there are $\prod_{i=1}^k(2^l-2^{k-1})$ such tuples in $\mathbb{F}_2^l$. To see this, note first that we have a choice of $2^l-1$ vectors for $v_1$. Given a choice $(v_1,\ldots,v_{j-1})$ of linearly independent vectors, there are $2^l-2^{j-1}$ choices for $v_j$ such that the set $\{v_1,\ldots,v_j\}$ is linearly independent, as clearly $|\mathrm{Span}(v_1,\ldots,v_{j-1})|=2^{j-1}$. The total number of $k$-tuples $(v_1,\ldots,v_k)$ of linearly independent vectors is therefore $(2^l-1)(2^l-2)\cdots(2^l-2^{k-1})=\prod_{i=1}^k(2^l-2^{i-1})$.

    We now need to take into account of the fact that several different $k$-tuples of linearly independent vectors will generate the same subspace. To this end, given such a subspace $V$, we note that this is precisely the same as asking how many different bases $V$ has. As $V$ is a $k$-dimensional vector space, this is the same as asking how many $k$-tuples of linearly independent vectors there are in $\mathbb{F}_2^k$. By the above argument, there are $\prod_{i=1}^k(2^k-2^{i-1})$ such $k$-tuples.

    The total number of $k$-dimensional subspaces of $\mathbb{F}_2^l$, and hence the number of subgroups of $(C_2)^l$ of order $2^k$, is therefore
    \[\frac{\prod_{i=1}^k(2^l-2^{i-1})}{\prod_{i=1}^k(2^k-2^{i-1})}=\prod_{i=1}^k\frac{2^l-2^{i-1}}{2^k-2^{i-1}}=\prod_{i=1}^k\frac{2^{l-(i-1)}-1}{2^i-1}.\]
\end{proof}
To obtain the total number of subgroups of $(C_2)^l$, we therefore must sum (\ref{k_c_2_subgroups}) over each order $2^k$ for $1 \leq k \leq l$, and add 1, which represents the trivial subgroup. For conciseness, we will extend the sum and take the value of the empty product to be 1. That is, for example, $\prod_{i=1}^0f(i):=1$ where $f(i)$ is any summand.
\begin{corollary}\label{c_2_subgroups}
    The number of subgroups of $(C_2)^l$ is
    \[\sum_{k=0}^l\prod_{i=1}^k\frac{2^{l-(i-1)}-1}{2^i-1}.\]
\end{corollary}
We now state two useful propositions from \cite{BML21}.
\begin{proposition}[Proposition 2.2 of \cite{BML21}]\label{ab-aut}
    Let $N$ be an abelian group such that $\Aut(N)$ is also abelian, and let $A$, $A'$ be subgroups of $\Aut(N)$. Consider the subgroups $M=N \rtimes A$ and $M'=N \rtimes A'$ of $\Hol(N)$. If there is an isomorphism $\phi: M \to M'$ with $\phi(N)=N$, then $M=M'$.
\end{proposition}
\begin{proposition}[Proposition 2.3 of \cite{BML21}]\label{rel-aut-char}
    Let $N$ be any group and $A$ be a subgroup of $\Aut(N)$. Let $M$ be the
    subgroup $N \rtimes A$ of $Hol(N)$, and suppose that $N$ is a characteristic subgroup of $M$. Then the group
    \[\Aut(M,A) := \{ \theta \in \Aut(M) : \theta(A)=A\}\] 
    is isomorphic to the normaliser of $A$ in $\Aut(N)$. In particular, if $\Aut(N)$ is abelian then $\Aut(M,A) \cong \Aut(N)$.
\end{proposition}
We finally recall the definition of Hall $\pi$-subgroups, as well as a related theorem. Let $G$ be a group with order divisible by $n$, we will denote by $\pi(n)$ the set of primes dividing $n$, and by $\pi'(n)$ the set of primes dividing $|G|$ which do not divide $n$. If $\pi$ is any set of primes dividing $|G|$, we recall that a \textit{$\pi$-subgroup} of $G$ is any subgroup of $G$, whose order is a product of primes in $\pi$ and a \textit{Hall $\pi$-subgroup} of $G$ is a $\pi$-subgroup of $G$ whose index is not divisible by any of the primes in $\pi$. We also recall Hall's Theorem below.
\begin{theorem}[Hall, 1928. See 9.1.7 of \cite{Rob96}]
    Let $G$ be a finite solvable group and $\pi$ any set of primes dividing $|G|$. Then every $\pi$-subgroup of $G$ is contained in a Hall $\pi$-subgroup of $G$. Further, any two Hall $\pi$-subgroups of $G$ are conjugate.
\end{theorem}

\section{Extensions of degree \texorpdfstring{$n$}{n} for \texorpdfstring{$n$}{n} a Cunningham product}
Let $n=p_1\cdots p_l$ be a Cunningham product. For each $1\leq i \leq l-1$, we have $p_i-1=2p_{i+1}$. Let $p_l-1=2^xs$ for some odd $s$. Further, let $F(l)$ denote the $l$\textsuperscript{th} Fibonacci number; we will use the convention that $F(2)=2$.
\begin{proposition}\label{fibbonacci}
    Let $n$ be as above. There are $F(l)$ groups of order $n$. These groups consist of a chain of ordered cyclic factors $..., C_{p_{i-1}}, C_{p_i}, C_{p_{i+1}},..$ such that each pair of consecutive factors is joined together via either a direct or semidirect product. Within each set of three consecutive factors, there is at most one semidirect product.
\end{proposition}
\begin{example}\label{fib_3_example}
    Let $n=p_1p_2p_3$, then we get the following $F(3)=3$ groups
    \begin{align*}
	& C_{p_1} \times C_{p_2} \times C_{p_3},\\
	& (C_{p_1} \rtimes C_{p_2}) \times C_{p_3},\\
	& C_{p_1} \times (C_{p_2} \rtimes C_{p_3}).
    \end{align*}
\end{example}
\begin{proof}[Proof of Proposition \ref{fibbonacci}]
    Firstly, as $n$ is squarefree, by \cite{MM84} we know that any group of order $n$ is metacyclic of the form
    \[C_e \rtimes C_d,\]
    where $n=ed$ is some factorisation of $n$ and $C_d$ acts faithfully on $C_e$. In particular, $C_e$ is normal. The elements of $C_d$ may therefore be identified with a subset of automorphisms of $C_e$, so $d$ must divide $\varphi(e)$. Thus if, say, $p_i \mid d$, then $p_i \mid \varphi(e)$, which is only possible if $p_{i-1} \mid e$, as we recall that $\varphi(p_{i-1})=2p_i$. Further, if $p_{i+1} \mid d$ as well, then by the same logic, we have $p_i \mid e$, which gives a contradiction, and thus we must have $p_{i-1}p_{i+1} \mid e$. We then obtain a $C_{p_i}$ factor of $C_d$ acting non-trivially on the $C_{p_{i-1}}$ factor of $C_e$. We may continue to choose primes which divide $d$ in this manner, which leads to the presentation in the statement of the proposition (indeed each prime dividing $d$ introduces a semidirect product between cyclic groups of `consecutive' prime order). It is clear to see that this does give us a group as it is the direct product of a cyclic group and several metacyclic groups, where each metacyclic group has the form $C_{p_i}\rtimes C_{p_{i+1}}$.
 
    The number of such groups can be obtained by using induction. Indeed, we note that there is just $F(1)=1$ group, $C_{p_1}$ of size $n=p_1$. For $n=p_1p_2$, we obtain the cyclic group $C_{p_1}\times C_{p_2}$ and the metacyclic group $C_{p_1}\rtimes C_{p_2}$, giving us $F(2)=2$ groups in total. Now, for some $m$, assume for all $k\leq m$, that there are $F(k)$ groups of order $n=p_1\cdots p_k$. Consider $n=p_1\cdots p_kp_{k+1}$. We may introduce a semidirect product between the last two cyclic factors; there are $k-1$ cyclic factors remaining, for which we know there are $F(k-1)$ such groups. If we choose instead to have a direct product between the last two cyclic factors, there are $F(k)$ such groups of this form. In total, we have $F(k)+F(k-1)=F(k+1)$ groups. Therefore, for $n=p_1\cdots p_l$, there are $F(l)$ groups of order $n$.
\end{proof}

\subsection{The cyclic group}\label{cyclic_cunningham}
Let $N$ be the cyclic group of order $n=p_1\cdots p_l$. This corresponds to the case in Proposition \ref{fibbonacci} where all actions are trivial. Suppose then that $N=\langle \sigma_1,\ldots,\sigma_l \rangle\cong C_n$, where $\sigma_i$ has order $p_i$. For $1 \leq i \leq l-1$, let $\alpha_i,\beta_i \in \Aut(\langle \sigma_i \rangle)$ have orders $p_{i+1},2$ respectively; further, let $\gamma,\delta \in \Aut(\langle \sigma_l \rangle)$ have orders $2^x,s$ respectively, where $p_l-1=2^xs$ with $s$ odd. Then the set of these elements generates $\Aut(N)$. We will write
\[\text{Hol}(N)=\langle \sigma_1,\ldots,\sigma_{l} \rangle \rtimes \left(\langle \alpha_1,\ldots,\alpha_{l-1}\rangle \times \langle \beta_1,\ldots,\beta_{l-1},\gamma \rangle \times \langle \delta \rangle \right),\]
which has order $2^{x+l-1}p_1p_2^2\cdots p_l^2s$. For ease of notation, we will sometimes denote $\delta$ by $\alpha_l$ and $\gamma$ by $\beta_l$. We now move towards finding and counting the subgroups of $\Aut(N)$.\\
It is clear that the subgroups of $\langle \alpha_1,\ldots,\alpha_{l-1} \rangle$ are of the form $\langle \alpha_1^{a_1},\ldots,\alpha_{l-1}^{a_{l-1}} \rangle$ where $a_i \in \left\{0,1\right\}$, as $\alpha_j$ and $\alpha_j$ have coprime orders for $i \neq j$; there are $2^{l-1}$ such subgroups.\\
Recall that $\sigma_0(s)$ denotes the number of divisors of $s$. The subgroups of $\langle \delta \rangle$ are of the form $\langle \delta^{s/d} \rangle$, where $d$ is some divisor of $s$; there are $\sigma_0(s)$ such subgroups.
\begin{proposition}\label{2-subgroups}
The number of subgroups of $\langle \beta_1,\ldots,\beta_{l-1},\gamma\rangle$ is
\[\Sigma_l:=\sum_{k=0}^{l-1}(2^{l-(k+1)}x+1)\prod_{i=1}^k\frac{2^{l-i}-1}{2^i-1}.\]
\end{proposition}
\begin{proof}
    Let $B:=\langle \beta_1,\ldots,\beta_{l-1},\gamma\rangle$ and $B':=\langle \beta_1,\ldots,\beta_{l-1}\rangle$. We first determine the form of an arbitrary subgroup $\langle z_1,\ldots, z_r \rangle$ of $B$. Suppose then that $z_i=\beta_1^{a_{i1}}\beta_2^{a_{i2}}\cdots\beta_{l-1}^{a_{il-1}}\gamma^{a_{il}}$. We can consider the following matrix
    \[\begin{pmatrix}
        a_{11} & a_{12} & \cdots & a_{1l}\\
        a_{21} & a_{22} & \cdots & a_{2l}\\
        \vdots & \vdots & \ddots & \vdots\\
        a_{r1} & a_{r2} & \cdots & a_{rl}
    \end{pmatrix}\]
    where each row represents a generator of the subgroup. The entries can be seen to be in $\mathbb{Z}/2^x\mathbb{Z}$, but we have $a_{ij} \in \{0,1\}$ for $j<l$. We note that applying elementary row operations does not change the subgroup, and so we may assume without loss of generality that $a_{ii}=\gcd(a_{ii},a_{(i+1)i},\ldots,a_{ri})$ and then that $a_{ij}=0$ for each $1 \leq i \leq l$ and $i<j \leq r$. We obtain the following $l \times l$ matrix, which is in row echelon form:
    \[\begin{pmatrix}
        a_{11} & a_{12} & \cdots & a_{1l}\\
        0 & a_{22} & \cdots & a_{2l}\\
        \vdots & \vdots & \ddots & \vdots\\
        0 & 0 & \cdots & a_{ll}
    \end{pmatrix}\]
    where we delete rows $l+1,\ldots,r$, which are all zero. Thus an arbitrary subgroup of $B$ has the form
\begin{equation}\label{subgroup}
    \langle \beta'_1\gamma^{2^{x-c_1}},\ldots,\beta'_{l-1}\gamma^{2^{x-c_{l-1}}},\gamma^{2^{x-c_l}}\rangle
\end{equation}
    for some fixed $0 \leq c_i \leq x$, $1\leq i \leq l$, and some $\beta'_j \in B'$, $1 \leq j \leq l-1$. We note further that if we have two generators of the subgroup $g_1:=\beta'_i\gamma^{2^{x-c_i}}, g_2:=\beta'_j\gamma^{2^{x-c_j}}$ such that $0 \leq c_i \leq c_j$, then $g_2^{-2^{c_j-c_i}}g_1=\beta'_i$, and so we may assume that $c_i=0$. More generally therefore, we may assume that $c_i=c_j$ for all nonzero $c_i,c_j$ in (\ref{subgroup}). Thus we now denote $c_i$ simply as $c$.
 
    We now count the subgroups $S$ of this form for a fixed $c$ with $1 \leq c \leq x$.

    (i$_c$) If $\gamma^{2^{x-c}}$ is a generator of $S$, then we can assume all the other generators involve only $\beta_1,\ldots, \beta_{l-1}$. So $S$ corresponds to a matrix of $l$ columns in reduced row echelon form (RREF) such that the first $l-1$ columns have entries in $\{0,1\}$ and the final column has all zeros apart from $2^{x-c}$ in the last entry.
   
    (ii$_c$) If $\gamma^{2^{x-c}}$ does not occur as a generator of $S$, then $S$ corresponds to a matrix of $l$ columns in RREF such that the first $l-1$ columns have entries in $\{0,1\}$ and entries in the final column are either $0$ or $2^{x-c}$, with $0$ in the last entry.
    
    We note that the cases (i$_c$) are disjoint from (i$_{c'}$) for $c \neq c'$ and from (ii$_{c''}$) for any $1 \leq c'' \leq x$. Further, we note that the matrix determines the value of $c$ except when the last column consists entirely of zeros; in this case the matrix corresponds to a subgroup of $B'$. Thus, given two values of $c$, say $c$ and $c'$, the cases (ii$_c$) and (ii$_{c'}$) are not disjoint, but intersect precisely when every entry of the final column of the matrix is equal to 0.
    
    Now, for each $c>0$, we have a bijection between the matrices in (i$_c$) and the matrices in (ii$_c$), and the matrices with $l$ columns over $\mathbb{F}_2$ in RREF by replacing each occurrence of $2^{x-c}$ by $1$, whenever it occurs. The number of these matrices is the same as the number of subgroups of $(C_2)^l$. That is, by Corollary \ref{c_2_subgroups}:
    \begin{equation}\label{l-RREF}
        \sum_{k=0}^l\prod_{i=1}^k\frac{2^{l+1-1}-1}{2^i-1}.
    \end{equation}
    To obtain the total number of subgroups of $B$, we note that we can just sum (\ref{l-RREF}) over $1 \leq c \leq x$, with the caveat that we over-count by
    \[(x-1)\sum_{k=1}^{l-1}\prod_{i=1}^k\frac{2^{l-i}-1}{2^i-1}.\]
    This is precisely due to the intersection between the two cases (ii$_c$) and (ii$_{c'}$) described above. In theory, the subgroups of $B'$ correspond to the case $c=0$, but this is already covered as discussed. So we need to take away $(x-1)$ times the number of subgroups of $B' \cong (C_2)^{l-1}$ from our initial count, as we need to count it exactly once. By Corollary \ref{c_2_subgroups}, the number of subgroups of $B'$ is precisely
    \[\sum_{k=1}^{l-1}\prod_{i=1}^k\frac{2^{l-i}-1}{2^i-1}.\]
    Therefore, the number of subgroups of $B$ is:
    \begin{equation}\label{subsofB}
        x\sum_{k=0}^l\prod_{i=1}^k\frac{2^{l+1-i}-1}{2^i-1}-(x-1)\sum_{k=0}^{l-1}\prod_{i=1}^k\frac{2^{l-i}-1}{2^i-1}.
    \end{equation}
    We claim that this simplifies to
    \[\sum_{k=0}^{l-1}(2^{l-(k+1)}x+1)\prod_{i=1}^k\frac{2^{l-i}-1}{2^i-1}.\]
    To see this, we may first split the sum (\ref{subsofB}) as follows:
    \begin{equation}\label{split_sum}
        x\left[\sum_{k=0}^l\prod_{i=1}^k\frac{2^{l+1-i}-1}{2^i-1}-\sum_{k=0}^{l-1}\prod_{i=1}^k\frac{2^{l-i}-1}{2^i-1}\right]+\sum_{k=0}^{l-1}\prod_{i=1}^k\frac{2^{l-i}-1}{2^i-1}.
    \end{equation}
    Noting that the product
    \[\prod_{i=1}^k\frac{2^{l+1-i}-1}{2^i-1}\]
    evaluates to $1$ when $k=l$, we have that the `$x$' part of (\ref{split_sum}) can be written as
    \[1+\sum_{k=0}^{l-1}\left[\prod_{i=1}^k\frac{2^{l+1-i}-1}{2^i-1}-\prod_{i=1}^k\frac{2^{l-i}-1}{2^i-1}\right].\]
    We further note that the summand for $k=0$ evaluates to $1-1=0$, and therefore we may start the sum from $k=1$. Considering, then, the difference between the two products in the summand, this evaluates to
    \[1+\sum_{k=1}^{l-1}\frac{2^l-2^{l-k}}{2^k-1}\prod_{i=1}^{k-1}\frac{2^{l-i}-1}{2^i-1}.\]
    Noting now that $(2^l-2^{l-k})/(2^k-1)=2^{l-k}$ and making the change of variables $k \mapsto k+1$, we see that this is equal to
    \[1+\sum_{k=0}^{l-2}2^{l-(k+1)}\prod_{i=1}^k\frac{2^{l-i}-1}{2^i-1}.\]
    Finally, noting that when $k=l-1$, the summand evaluates to $1$, and hence the `$x$' part of the sum can be written as
    \[\sum_{k=1}^{l-1}2^{l-(k+1)}\prod_{i=1}^k\frac{2^{l-i}-1}{2^i-1}.\]
    The rest is clear.
\end{proof}
\begin{lemma}
    $\Hol(N)$ contains a unique Hall $\pi$-subgroup $H$, where $\pi$ is the set of primes dividing $n$.
\end{lemma}
\begin{proof}
    It is easy to see that the subgroup
    \[H:=\langle \sigma_1,\ldots, \sigma_l, \alpha_1,\ldots,\alpha_{l-1} \rangle\]
    is a normal Hall $\pi$-subgroup of $\Hol(N)$.
\end{proof}
\begin{corollary}
    If $M\leq\Hol(N)$ is transitive on $N$, then $M\cap H$ must also be transitive on $N$.
\end{corollary}
\begin{proof}
    This is adapted from the proof of Lemma 5 of \cite{Cre22}. Consider the index of $(M \cap H) \cap \mathrm{Stab}_{\Hol(N)}(1_N)$ in $M$; we note that we can factorise it in the following two ways:
    \begin{align*}
        &[M:(M \cap H) \cap \mathrm{Stab}_{\Hol(N)}(1_N)] =\\
        &=[M:M \cap H][M \cap H: (M \cap H) \cap \mathrm{Stab}_{\Hol(N)}(1_N)]\\
        &=[M:M \cap \mathrm{Stab}_{\Hol(N)(1_N)}][M \cap \mathrm{Stab}_{\Hol(N)(1_N)}:(M \cap H) \cap \mathrm{Stab}_{\Hol(N)}(1_N)].
    \end{align*}
    Now, as $H$ is the unique Hall $\pi$-subgroup of $\Hol(N)$, then $M \cap H$ is a Hall $\pi$-subgroup of $M$, and so $[M:M \cap H]$ is coprime with $|N|$. We also have that
    \[[M \cap \mathrm{Stab}_{\Hol(N)(1_N)}:(M \cap H) \cap \mathrm{Stab}_{\Hol(N)}(1_N)]=\frac{|M \cap \mathrm{Stab}_{\Hol(N)(1_N)}|}{|(M \cap H) \cap \mathrm{Stab}_{\Hol(N)}(1_N)|}=\frac{|\mathrm{Stab}_M(1_N)|}{|\mathrm{Stab}_{M \cap H}(1_N)|}\]
    is coprime with $|N|$. Therefore, we must have that
    \[|\mathrm{Orb}_{M \cap H}(1_N)|=[M \cap H: (M \cap H) \cap \mathrm{Stab}_{\Hol(N)}(1_N)]=|N|\]
    and hence $M \cap H$ is transitive on $N$.
\end{proof}
For the next Proposition, we recall that $[\eta,\alpha]$ denotes an element of $\Hol(N)=N \rtimes \Aut(N)$, where $\eta \in N$ and $\alpha \in \Aut(N)$.
\begin{proposition}\label{H_subgroups}
    The subgroups $M$ of $H$ which are transitive on $N$ are of the form
    \[\langle \sigma_1 \rangle \times A_2 \times \cdots \times A_l,\]
    where, for each $2 \leq i\leq l$, $A_i$ is either $\left\langle \left[\sigma_i,\alpha_{i-1}^{t_i}\right] \right\rangle$ of order $p_i$ or $\left\langle \sigma_i,\alpha_{i-1} \right\rangle$ of order $p_i^2$, with $0\leq t_i \leq p_i-1$.
    
    For each $2 \leq i \leq l-1$, if both $A_i$ and $A_{i+1}$ have prime order, then either $A_i=\langle \sigma_i \rangle$ or $A_{i+1}=\langle \sigma_{i+1} \rangle$.
\end{proposition}
\begin{proof}
    We may write
    \[H=\left\langle \sigma_1 \right\rangle \times \left\langle \sigma_2,\alpha_1 \right\rangle \times \cdots \times \left\langle \sigma_l, \alpha_{l-1} \right\rangle.\]
    The first factor has order $p_1$, and each subsequent factor has order $p_i^2$ for $2\leq i \leq l$. Therefore the subgroups $M$ of $H$ are direct products of subgroups of each of the factors. The subgroups of $\langle \sigma_1 \rangle$ are $\{1\}$ and $\langle \sigma_1 \rangle$, and the subgroups of each factor $\left\langle \sigma_i,\alpha_{i-1} \right\rangle$ are $\{1\}$, $\langle \alpha_{i-1} \rangle$, $\left\langle \left[\sigma_i, \alpha_{i-1}^{t_i}\right] \right\rangle$ for $0 \leq t_i \leq p_i-1$, and $\left\langle \sigma_i,\alpha_{i-1} \right\rangle$.
	
    Note that if any of the factors of $M$ are trivial, then $n \nmid |M|$, so $M$ cannot be transitive on $N$. Further, if, for any $2 \leq i \leq l$, the $i$\textsuperscript{th} factor is $\langle \alpha_{i-1} \rangle$, then $M$ is not transitive as there is no element in $M$ sending $1_N$ to $\sigma_i \in N$. It is clear that if all the factors are of the other types, then $M$ will be transitive. We finally need to check that what we have is indeed a group. Note that any two non-adjacent factors commute, so now suppose that there are adjacent factors of the form $\left\langle\left[\sigma_i,\alpha_{i-1}^{t_i}\right]\right\rangle$ (say, for $i=j,j+1$). The factor for $i=j+1$ must normalise the factor for $i=j$. We have
    \[\left[\sigma_{j+1},\alpha_j^{t_{j+1}}\right]\left[\sigma_j,\alpha_{j-1}^{t_j}\right]\left[\sigma_{j+1},\alpha_j^{t_{j+1}}\right]^{-1}=\left[\sigma_j^{a_j^{t_{j+1}}},\alpha_{j-1}^{t_j}\right],\]
    which is in $\left\langle\left[\sigma_j,\alpha_{j-1}^{t_j}\right]\right\rangle$ if and only if either $t_j=0$ or $t_{j+1}=0$. Therefore, if the factors $A_i$ and $A_{i+1}$ have prime order, then either $A_i=\langle \sigma_i \rangle$ or $A_{i+1}=\langle \sigma_{i+1} \rangle$.
\end{proof}
\begin{corollary}
    Suppose $L/K$ is Galois of degree $n$, where $n$ is an $l$-Cunningham product for some $l$. Then $L/K$ admits a Hopf--Galois structure of cyclic type.
\end{corollary}
\begin{proof}
    We see that all the groups in Proposition \ref{fibbonacci} appear in the list in Proposition \ref{H_subgroups}. We may identify the factor $C_{p_{i-1}} \rtimes C_{p_i}$ with the subgroup
    \[\left\langle \sigma_{i-1}, \left[\sigma_i,\alpha_{i-1}^{t_i}\right] \right\rangle,\]
    and identify the factor $C_{p_{i-1}} \times C_{p_i}$ with the subgroup $\langle \sigma_{i-1}, \sigma_i \rangle$.
\end{proof}
In light of Proposition \ref{H_subgroups}, we give the following definitions:
\begin{definition}
    Let $I:=\{i_1,\ldots,i_u\}$ for some $u \leq l$ be a subset of $\{2,\ldots,l\}$ such that no two consecutive integers are contained in $I$.

    Given such an $I$, let $\bt:=(t_{i_j})_{i_j \in I}$ be the sequence of length $|I|$ such that $1 \leq t_{i_j} \leq p_{i_j}-1$ for each $t_{i_j}$ in $\bt$. For ease of notation, we sometimes drop the parentheses, and we denote the empty sequence by $\emptyset$.

    Finally, given an $I$ and a corresponding $\bt$, define the groups $J_{I,\bt}$ by
    \[J_{I,\bt}:=\left\langle \sigma_j, \left[\sigma_i,\alpha_{i-1}^{t_i}\right] \mid i \in I, j \in \{1,\ldots,l\}\setminus I \right\rangle,\]
    and the group $N_I \leq J_{I,\bt}$ by:
    \[N_I:=\langle \sigma_j \mid j \in \{1,\ldots,l\} \setminus I \rangle.\]
\end{definition}
\begin{remark}\label{gps_order_n}
    Let $\mathcal{I}$ be the set of all sets $I$ described above, and
    \[\mathcal{J}:=\{J_{I,\bt_1(I)} \mid I \in \mathcal{I}\},\]
    where $\bt_1(I)$ is the sequence of length $|I|$ containing all 1's. Then it is clear that there is a bijection between $\mathcal{J}$ and the set of abstract isomorphism classes of the groups of order $n$.
\end{remark}
\begin{example}\label{subset_examples}
    First if $n$ is an arbitrary $l$-Cunningham product for some $l$ with $I=\emptyset$, then $\bt=\emptyset$, so  $J_{\emptyset,\emptyset}\cong C_n$.
    
    Now suppose $n=p_1p_2$ and $I=\{2\}$ with $\bt=t$, then $J_{I,\bt} = J_{\{2\},t}= \left\langle \sigma_1, \left[\sigma_2,\alpha_1^t\right]\right\rangle$; this is the group $J_t \cong C_{p_1} \rtimes C_{p_2}$ defined in \cite{BML21}.
\end{example}
We further define the set
\[S_I:=\{i \mid i-1 \text{ and } i \in \{1,\ldots,l\}\setminus I\},\]
along with the subgroup $A_I \leq \Aut(N_I)$ by
\[A_I:=\langle \alpha_i,\beta_j \mid i \in S_I, j \in \{1,\ldots,l\} \setminus I \rangle.\]
\begin{proposition}\label{cyclic_trans}
    The transitive subgroups of $\Hol(N)$ are of the form
    \[J_{I,\bt} \rtimes A,\]
    where $A$ is any subgroup of $A_I$.
\end{proposition}
\begin{proof}
    We note that every transitive subgroup listed in Proposition \ref{H_subgroups} contains a regular normal subgroup, which has the form $J_{I,\bt}$ for some $I$ and $\bt$. The normaliser of $J_{I,\bt}$ in $\Aut(N)$ is precisely $\Aut(N_I)$, since if $\phi \in \Aut(N)\setminus \Aut(N_I)$, then $\phi(\sigma_i)\neq \sigma_i$ for some $i \in I$, and so $\phi \left[\sigma_i,\alpha_{i-1}^{t_i}\right]\phi^{-1}=\left[\phi(\sigma_i),\alpha_{i-1}^{t_i}\right]\notin J_{I,\bt}$. Therefore $J_{I,\bt}$ may be extended by any subgroup $A$ of $\Aut(N_I)$ to give a transitive subgroup of $\Hol(N)$, and any transitive subgroup is obtained in this way.

    We note, however, that, given an $I,\bt$ and $A \leq \Aut(N_I)$, then if $\alpha_{i-1} \in A$ for some $i \in I$, we have $J_{I,\bt} \rtimes A = J_{I',\bt'} \rtimes A$, where $I'=I\setminus\{i\}$ and $\bt'$ is the sequence $\bt$ with the term $t_i$ removed. Thus in order to list the transitive subgroups with no repeats, we must have $A \leq A_I$.
\end{proof}
    For each group $J_{I,\bt} \rtimes A$, the stabiliser, $A$, of $1_N$ has normal complement, $J_{I,\bt}$, and so all the corresponding field extensions are almost classically Galois. We now explore which groups are isomorphic as permutation groups.
\begin{proposition}\label{cyclic_isoms_cunningham}
    For a fixed $I$ and $A$, the $\prod_{i \in I}(p_i-1)$ transitive groups of the form $J_{I,\bt} \rtimes A$ listed in Proposition \ref{cyclic_trans} are isomorphic as permutation groups. There are no other isomorphisms, even as abstract groups.
\end{proposition}
\begin{proof}
    Firstly, consider two subsets $I \neq I'$, with $A,A'$ such that $A\leq A_I$, $A'\leq A_{I'}$, and fix the appropriate $\bt,\bt'$ so that $J_{I,\bt}\rtimes A$ and $J_{I',\bt'} \rtimes A'$ are distinct transitive subgroups. Then, if $I \not\subset I'$, we have that $J_{I,\bt}\rtimes A$ contains an abelian subgroup of order $\prod_{i\in \{1,\ldots, l\}\setminus I}p_i$, namely $N_I$, but there is no abelian subgroup of this order in $J_{I',\bt'}\rtimes A'$, and so they cannot be isomorphic as abstract groups. If $I \subset I'$, then we may apply the same argument on $J_{I',\bt}\rtimes A'$ with $N_I$ replaced by $N_{I'}$. Next, for a subset $I$ and a fixed choice for $A$, we show that the groups of the form $J_{I,\bt} \rtimes A$ are isomorphic; indeed, choose a $\phi \in \Aut(N)$ such that $\phi(\sigma_i)=\sigma_i^{t_i}$ for each $i \in I$, then
    \[\phi\left[\sigma_i,\alpha_{i-1}^{t_i}\right]\phi^{-1}=\left[\phi(\sigma_i),\alpha_{i-1}^{t_i}\right]=\left[\sigma_i,\alpha_{i-1}\right]^{t_i}.\]
    Therefore conjugation by $\phi$ gives an isomorphism between $J_{I,\bt}$ and $J_{I,\bt_1(I)}$, where $\bt_1(I)$ is defined as in Remark \ref{gps_order_n}. For any appropriate $A$, this isomorphism extends to an isomorphism $J_{I,\bt}\rtimes A\to J_{I,\bt_1(I)} \rtimes A$. This isomorphism fixes the stabiliser $A$ of $1_N$, and so the groups are isomorphic as permutation groups.
 
    Suppose now that there is an isomorphism $\phi$ between the groups $J_{I,\bt} \rtimes A$ and $J_{I,\bt} \rtimes A'$, then we claim that $A=A'$; we modify the proof of Proposition $2.2$ in \cite{BML21}. Note that $J_{I,\bt}$ contains the characteristic abelian group $N_I$, and that $A,A'$ are both subgroups of $\Aut(N_I)$, which is abelian. For $g \in \Hol(N)$, write $C_g$ for conjugation by $g$. Then take $\alpha \in A$ such that $\phi(\alpha)=[\eta,\alpha']$ where $\eta \in N$ and $\alpha' \in A'$ (we may infer that $\alpha \in A'$ by comparing orders of elements and from the fact that $N_I$ is characteristic). For $\mu \in N_I$, we have $C_{\alpha}(\mu)=\alpha\mu\alpha^{-1}$, so applying $\phi$ to this, we obtain
    \[\phi(C_{\alpha}(\mu))=\phi(\alpha)\phi(\mu)\phi(\alpha)^{-1}=C_{\phi(\alpha)}(\phi(\mu)).\]
    Therefore, within $\Aut(N_I)$, we have $\phi C_{\alpha} = C_{\phi(\alpha)}\phi$, where $\phi$ is now viewed as an element of $\Aut(N_I)$. As $\Aut(N_I)$ is abelian, we have that $C_{\alpha}=C_{\phi(\alpha)}$, and so
    \[\alpha=C_{\alpha}=C_{\phi(\alpha)}=C_{\eta}C_{\alpha'}=C_{\alpha'}=\alpha'\]
    as $C_{\eta}=1_{N_I}=1_N$ and $C_{\alpha}=\alpha$ by the multiplication in $\Hol(N)$. This proves the claim.
\end{proof}
\begin{proposition}\label{cyclic_HGS_Cunningham}
    Let $M:=J_{I,\bt}\rtimes A$ be a transitive subgroup of $\Hol(N)$. Then any separable extension associated to the isomorphism class of $M$ admits
    \[\prod_{i \in I}p_{i-1}^{y_{i-1}}\]
    Hopf--Galois structures of cyclic type, where $y_{i-1}=1$ whenever $\Aut(\langle \sigma_{i-1} \rangle) \cap A = \{1\}$, and $0$ otherwise.
\end{proposition}
\begin{proof}
    Consider first the groups of the form $M=J_{I,\bt}\rtimes A$. We note that we may apply Proposition \ref{rel-aut-char} with $J_{I,\bt}$ in place of $N$ and $M'=A$. We may do this because the stabiliser $M'$ of $1_N$ in $M$ is $A$, which is abelian, and $J_{I,\bt}$ is characteristic in $M$ as $N_I$ is characteristic and $\gcd(|J_{I,\bt}/N_I|,|A|)=1$. Conjugation by any non-identity element of $A$ acts non-trivially on some $\sigma_j\in N_I$ and fixes the group $J_{I,\bt}/N_I$. If $\phi \in \Aut(J_{I,\bt})$, we have $\phi(\sigma_j)=\sigma_j^{a_j}$ for all $j \in \{1,\ldots, l\}\setminus I$, and
    \[\phi\left(\left[\sigma_i,\alpha_{i-1}^{t_i}\right]\right)=\sigma_{i-1}^{a_{i-1}}\left[\sigma_i,\alpha_{i-1}^{t_i}\right]\]
    for each $i \in I$, where $1\leq a_j \leq p_j-1$ and $0\leq a_{i-1} \leq p_{i-1}$. Then $\phi$ normalises $A$ in $\Aut(J_{I,\bt})$ if and only if $a_{i-1}=0$ whenever $\Aut(\langle \sigma_{i-1} \rangle) \cap A \neq \{1\}$, $i \in I$. We therefore compute
    \[|\Aut(M,M')|=\prod_{i \in \{1,\ldots,l\}\setminus I}(p_i-1)\prod_{i \in I}p_{i-1}^{y_{i-1}}\] where $y_{i-1}=1$ whenever $a_{i-1}$ is allowed to be non-zero, and $0$ otherwise. There are therefore
    \[\frac{\prod_{i \in \{1,\ldots,l\}\setminus I}(p_i-1)\prod_{i \in I}p_{i-1}^{y_{i-1}}}{|\Aut(N)|}\times \prod_{i \in I}(p_i-1)=\prod_{i \in I}p_{i-1}^{y_{i-1}}\]
    Hopf--Galois structures of this type.
\end{proof}
\begin{theorem}\label{cyclic_cunningham_thm}
    There are in total
    \[\sum_{I \ni l}2^{l-2|I|}\left(\sum_{k=0}^{l-|I|}\prod_{i=1}^k\frac{2^{l-|I|-(i-1)}-1}{2^i-1}\right)+\sum_{I \not\ni l}2^{l-1-2|I|}\Sigma_{l-|I|}\sigma_0(s)\]
    isomorphism types of permutation groups $G$ of degree $n$ which are realised by a Hopf--Galois structure of cyclic type.
\end{theorem}
    The first sum is taken over all subsets $I$ for which $l \in I$, and the second sum is taken over all subsets $I$ for which $l \notin I$. We note that $\Sigma_{l-|I|}$ is as defined in Proposition \ref{2-subgroups}.
\begin{proof}
    For each $I$, we count the number of choices for $A$ such that $J_{I,\bt}\rtimes A$ is a transitive subgroup described in Proposition \ref{cyclic_trans}. Recall that $A$ is any subgroup of
    \[A_I=\langle \alpha_i,\beta_j \mid i \in S_I,j \in \{1,\ldots,l\} \setminus I \rangle.\]
    If $l \in I$, then there are $2^{l-2|I|}$ subgroups of the factor of $A_I$ generated by the $\alpha_i$'s, and there are
     \[\sum_{k=1}^{l-|I|}\prod_{i=1}^k\frac{2^{l-|I|-(i-1)}-1}{2^i-1}\]
    subgroups of the factor of $A_I$ generated by the $\beta_j$'s. If $l \notin I$, then $\gamma,\delta \in A_I$, and so there are
    \[2^{l-1-2|I|}\Sigma_{l-|I|}\sigma_0(s)\]
    subgroups of $A_I$.
\end{proof}
\begin{remark}\label{cyclic_SG}
    Taking $l=2$, we have $n=p_1p_2$ with $p_2=2p_1+1$. There are two possible subsets $I$ to consider, these are $I_1=\emptyset$ and $I_2=\{2\}$ (corresponding to the groups mentioned in Example \ref{subset_examples}). We thus compute
    \begin{align*}
	&2\left(\sum_{k=0}^1\prod_{i=1}^k\frac{2^{2-i}-1}{2^i-1}\right)+2\Sigma_2\sigma_0(s)\\
	&=2+2\Sigma_2\sigma_0(s).
    \end{align*}
    We have
    \[\Sigma_2=\sum_{k=0}^1(2^{1-k}x+1)\prod_{i=1}^k\frac{2^{2-i}-1}{2^i-1}=3x+2.\]
    We therefore recover Theorem 4.4 of \cite{BML21} that there are
    \[2+2(3x+2)\sigma_0(s)=(6x+4)\sigma_0(s)+2\]
    isomorphism types of permutation groups $G$ of degree $p_1p_2$ which are realised by a Hopf--Galois structure of cyclic type.
\end{remark}

\subsection{The metacyclic groups}\label{metab-cunningham}
In this section, we consider $N$ of non-abelian type. There are $F(l)-1$ non-abelian groups of order $n=p_1p_2\cdots p_l$, with structure described in Proposition \ref{fibbonacci} where we introduce at least one non-trivial semidirect product between the factors. To fix an $N$, fix a non-empty subset $I \subset \{2,\ldots,l\}$ such that $I$ does not contain any two consecutive integers, and define
\[m:=\frac{n}{\prod_{i \in I}p_{i-1}p_i}.\]
For each $i \in I$, let $k_i$ have order $p_i \pmod{p_{i-1}}$. We may then define
\[N=\langle \sigma_1,\ldots, \sigma_l \mid \sigma_i\sigma_{i-1}=\sigma_{i-1}^{k_i}\sigma_i \; \forall i \in I \rangle \cong C_m \times \prod_{i \in I}(C_{p_{i-1}} \rtimes C_{p_i})\]
where $\sigma_j$ has order $p_j$ and two elements commute whenever no relation is given between them. For notational convenience, for each $i \in I$, we identify the subgroup $\langle \sigma_{i-1},\sigma_i \rangle$ with the factor $C_{p_{i-1}} \rtimes C_{p_i}$ and the subgroup $\langle \{\sigma_1,\ldots, \sigma_l\} \setminus \{\sigma_{i-1},\sigma_i \mid i \in I\} \rangle$ with the factor $C_m$. We make some observations:
\begin{remark}\label{cunningham_metacyclic_observations}\phantom{boo}
    \begin{enumerate}[label=$\bullet$]
        \item We have the following decomposition:
        \[\Hol(N)=\Hol\left(C_m \times \prod_{i \in I}C_{p_{i-1}} \rtimes C_{p_i}\right)=\Hol(C_m) \times \prod_{i \in I} \Hol(C_{p_{i-1}}\rtimes C_{p_i}).\]
        This is because $\Aut(A \times B)=\Aut(A)\times\Aut(B)$ whenever $A$ and $B$ have coprime orders.
        \item A subgroup $M\leq\Hol(N)$ is transitive on $N$ if and only if it is transitive on each of the factors. Thus $M$ is transitive if and only if it is transitive on $C_m$ and on $C_{p_{i-1}}\rtimes C_{p_i}$ for each $i \in I$. Further, the results of Section \ref{cyclic_cunningham} together with Section 4.2 of \cite{BML21} tell us precisely the transitive subgroups on each factor of $N$.
        \item We will next see that the factors of $\Hol(N)$ given by the first point do not have coprime orders. In particular, their orders are all divisible by $2$, and, depending on the chosen $I$, the factors $\Hol(C_m)$ and $\Hol(C_{p_{i-1}}\rtimes C_{p_i})$ may both have orders divisible by $p_{i-1}$.
    \end{enumerate}
\end{remark}
The last remark highlights that we need to understand the different ways the transitive subgroups of each factor of $\Hol(N)$ lift to obtain a transitive subgroup of $\Hol(N)$. We now give a description of $\Aut(C_m)$ and $\Aut(C_{p_{i-1}}\rtimes C_{p_i})$ for each $i \in I$.

We first describe $\Aut(C_m)$. This is similar to Section \ref{cyclic_cunningham}, but we must also account for the fact that we may have the situation where $p_j \mid m$, $p_{j+1},p_{j+2} \nmid m$ but $p_{j+3} \mid m$ for some primes $p_j,\ldots,p_{j+3}$ in the Cunningham chain. Therefore, for each $\sigma_j \in \{\sigma_1,\ldots,\sigma_{l-1}\}\setminus\{\sigma_{i-1},\sigma_i \mid i \in I\}$, let $\alpha_j,\beta_j \in \Aut(\langle\sigma_j\rangle)$ have orders $p_{j+1},2$ respectively, and if $l \notin I$, let $\alpha_l,\beta_l \in \Aut(\langle\sigma_l\rangle)$ have orders $s,2^x$ respectively. Thus
\[\Aut(C_m) = \langle \alpha_j,\beta_j \mid p_j \text{ divides } m \rangle,\]
and has order
\[
\varphi(m)=
\begin{cases*}
    2^{l+x-(1+2|I|)}s\prod_{p_i \mid m}p_{i+1}, & when $l \notin I$,\\
    2^{l-2|I|}\prod_{p_i \mid m}p_{i+1}, & when $l \in I$,
\end{cases*}
\]
Where $\varphi$ is the Euler $\varphi$-function. We now describe $\Aut(C_{p_{i-1}} \rtimes C_{p_i})$ for each $i \in I$, which we recall from \cite{BML21}. We have that $\Aut(C_{p_{i-1}} \rtimes C_{p_i})$ is generated by the elements $\phi_i,\psi_i,\theta_i$ of orders $p_i,2,p_{i-1}$ respectively such that
\begin{align*}
	&\phi_i(\sigma_{i-1})=\sigma_{i-1}^{k_i},
        &&\phi_i(\sigma_i)=\sigma_i,\\
	&\psi_i(\sigma_{i-1})=\sigma_{i-1}^{-1},
	&&\psi_i(\sigma_i)=\sigma_i,\\
	&\theta_i(\sigma_{i-1})=\sigma_{i-1},
	&&\theta_i(\sigma_i)=\sigma_{i-1}\sigma_i.	
\end{align*}
Thus each factor $\Hol(C_{p_{i-1}}\rtimes C_{p_i})$ has order $2p_{i-1}^2p_i^2$.
\begin{remark}\label{commuting_elements}
    We make a similar observation as in \cite{BML21}, that $P_{i-1}:=\langle \sigma_{i-1},\theta_i \rangle$ is the unique Sylow $p_{i-1}$-subgroup of $\Hol(C_{p_{i-1}}\rtimes C_{p_i})$, which has order $p_{i-1}^2$, and has complementary subgroup $R_i:=\langle \sigma_i,\phi_i,\psi_i \rangle$ of order $2p_i^2$. Thus we may write
    \[\Hol(C_{p_{i-1}}\rtimes C_{p_i})=P_{i-1}\rtimes R_i=\langle \sigma_{i-1},\theta_i \rangle \rtimes \langle \sigma_i,\phi_i,\psi_i \rangle.\]
    We further similarly note that
    \[[\sigma_{i-1},\theta_i^{k_i-1}]\sigma_i=[\sigma_{i-1}\theta_i^{k_i-1}(\sigma_i),\theta_i^{k_i-1}]=[\sigma_{i-1}^{k_i}\sigma_i,\theta_i^{k_i-1}]=\sigma_i[\sigma_{i-1},\theta_i^{k_i-1}],\]
    and so $[\sigma_{i-1},\theta_i^{k_i-1}]$ commutes with $\sigma_i$.
\end{remark}
Given this remark, and following the method in \cite{BML21}, if we identify $P_{i-1}$ with the vector space $\mathbb{F}_{p_{i-1}}^2$ and choose the basis $\sigma_{i-1}$, $[\sigma_{i-1},\theta^{k_i-1}]$, which we identify with the vectors $\e_1:=(1,0)^t$ and $\e_2:=(0,1)^t$ respectively, then we may further identify $R_i$ with the subgroup of $\mathrm{GL}_2(\mathbb{F}_p)$ generated by the $\mathbb{F}_{p_{i-1}}$-linear maps
\begin{align*}&A_i:=\begin{pmatrix}
		k_i & 0	\\
		0 	& k_i
	\end{pmatrix},
	&&B_i:=\begin{pmatrix}
		-1	& 0	\\
		0	& -1
	\end{pmatrix},
	&&&S_i:=\begin{pmatrix}
		k_i & 0 \\
		0	& 1
	\end{pmatrix}
\end{align*}
via $\sigma_i \mapsto S_i$, $\phi_i \mapsto A_i$ and $\psi_i \mapsto B_i$. We will redefine $P_{i-1}:=\langle \e_1,\e_2 \rangle$ and $R_i:=\langle S_i,A_i,B_i \rangle$. We remark that $\e_1$ and $\e_2$ also depend on $i$, but it will almost always be clear what we mean, so we omit this dependence until needed.

We now define the two maps $\pi$ and $\omega$ such that
\[\pi: \Hol(N) \to \Hol(C_m) \to \Aut(C_m)\]
is the projection map to $\Aut(C_m)$, and
\[\omega: \Hol(N) \to R_i\]
is the quotient map to the factor $R_i= \langle S_i,A_i,B_i \rangle$.

The subgroup
\begin{equation}\label{B}
    B:=\langle \beta_i,B_j \mid i \notin I, j \in I \rangle
\end{equation}
is the unique Sylow $2$-subgroup of $\Hol(N)$ contained in $\Aut(N)$. Therefore any Sylow $2$-subgroup of $\Hol(N)$ must be conjugate to $B$, and is hence of the form
\[B(\mathbf{x},\mathbf{a},\mathbf{b}):=\langle [\sigma_i^{x_i},\beta_i],[a_j\e_{j_1}+b_j\e_{j_2},B_j] \mid i \notin I, j \in I \rangle,\]
where $\mathbf{x}:=(x_i)_{i \notin I}$, $\mathbf{a}:=(a_j)_{j \in I}$ and $\mathbf{b}:=(b_j)_{j \in I}$, with ranges $0 \leq x_i \leq p_i-1$ and $0 \leq a_j,b_j \leq p_j-1$. The elements $\e_{j_1}$ and $\e_{j_2}$ are the basis vectors of $\mathbb{F}^2_{p_{j-1}}$.

Let $M$ be a transitive subgroup of $\Hol(N)$ and let $H$ be a Sylow $2$-subgroup of $M$. Then $H=M \cap B(\mathbf{x},\mathbf{a},\mathbf{b})$ for some choice of $\mathbf{x},\mathbf{a},\mathbf{b}$.
\begin{proposition}\label{2-groups}
    Let $M$ be a transitive subgroup of $\Hol(N)$. Then $M$ has Sylow $2$-subgroup $H$ of the form
    \[H=M \cap B((0)_{i \notin I},(a_j)_{j \in I},(-a_j)_{j \in I}).\]
\end{proposition}
\begin{proof}
    By the discussion above, we know that every Sylow $2$-subgroup $H$ of $M$ has the form $H=M \cap B(\mathbf{x},\mathbf{a},\mathbf{b})$ for some $\mathbf{x},\mathbf{a},\mathbf{b}$. However, if $\beta_i \in \pi(M)$ for some $i \notin I$, then by the discussion in Section \ref{cyclic_cunningham}, we see that $\sigma_i \in M$. Similarly, for each $j \in I$ such that $B_j$ is in the image of the projection $M \to \Aut(N)$, we have either $\e_{j_1} \in M$ or $\e_{j_2} \in M$. Given $H$, therefore, we can thus find a Sylow $2$-subgroup of $M$ of the form described in the proposition.

    For example, suppose that the element $[a\e_1+b\e_2,B_j]$ is in the image of the projection $\pi':M \to \Hol(C_{p_{i-1}} \rtimes C_{p_i})$ and that, say, $\e_1 \in M$. Then we see that $[-b(\e_1-\e_2),B_j]$ is also in $\pi'(M)$.
\end{proof}
\begin{corollary}\label{2-isoms}
    Every transitive subgroup of $\Hol(N)$ is isomorphic, as a permutation group, to a transitive subgroup $M \leq \Hol(N)$ which contains a Sylow $2$-subgroup $H$ of the form $H=M \cap B$.
\end{corollary}
\begin{proof}
    Let $G$ be a transitive subgroup of $\Hol(N)$. Then by Proposition \ref{2-groups}, $G$ contains a Sylow $2$-subgroup of the form
    \[H:=G \cap B((0)_{i \notin I},(a_j)_{j \in I},(-a_j)_{j \in I}).\]
    For each $j \in I$, we note that the element $\e_{j_1}-\e_{j_2}$ is in $\Aut(N)$, and so conjugation by $a_j(\e_{j_1}-\e_{j_2})$ gives an isomorphism of permutation groups. The image of $H$ under conjugation by $a_j(\e_{j_1}-\e_{j_2})$ for each $j \in I$ is a group of the form $M \cap B$ for some transitive subgroup $M \leq \Hol(N)$.
\end{proof}
\begin{proposition}\label{cunningham_2gp_count}
    If $l \in I$, there are
    \[\sum_{k=0}^{l-|I|}\prod_{i=1}^k\frac{2^{l-|I|-(i-1)}-1}{2^i-1}\]
    subgroups of $B$.
    Otherwise, there are 
    \[\Sigma_{l-|I|}\]
    subgroups of $B$.
\end{proposition}
\begin{proof}
    See Propositions \ref{c_2_subgroups} and \ref{2-subgroups}.
\end{proof}
Now if, for some $i \in I$, we also have $i-2 \in I$, then $P_{i-1}=\langle \e_1,\e_2 \rangle$ is the unique Sylow $p_{i-1}$-subgroup of $\Hol(N)$, and so, by Lemma 4.5 of \cite{BML21}, $M$ contains either $\langle \e_1 \rangle, \langle \e_2 \rangle$, or $P_{i-1}$ as a Sylow $p_{i-1}$-subgroup. However, if $i-2 \notin I$, then there is an element $\alpha_{i-2} \in \Hol(C_m)$ of order $p_{i-1}$. Proposition \ref{p_i-groups} describes the different situations that could arise, but we first need the following lemma:
\begin{lemma}\label{coprime_elt}
    Let $M \leq \Hol(N)$ be transitive. Then $M$ contains an element of the form $(\beta,[\mathbf{v},SA_i^aB_i^b])$, where $0 \leq a \leq p_i-1$, $0 \leq b \leq 1$, $\mathbf{v} \in P_{i-1}$, and $\beta \in \Hol(N)$ is an element of order $2$ which commutes with $\alpha_{i-2}$ and the factor $\Hol(C_{p_{i-1}} \rtimes C_{p_i})$ of $\Hol(N)$.
\end{lemma}
\begin{proof}
    The subgroup $\langle S_i, A_i \rangle$ is a Sylow $p_i$-subgroup of $\Hol(N)$. Thus any transitive subgroup $M\leq \Hol(N)$ must contain a subgroup conjugate to a subgroup of $\langle S_i, A_i \rangle$, and hence an element of the form $[\mathbf{v},S_i^aA_i^b]$ for some $0 \leq a,b \leq p_i-1$ and $\mathbf{v} \in P_{i-1}$. We note that we may choose an element such that $a>0$, as if there is no such element with $a>0$, then $M$ cannot be transitive.

    By Proposition \ref{2-groups}, $M$ contains an element of the form $(\beta,[\mathbf{v}',B_i^c])$, where $0 \leq c \leq 1$, $\mathbf{v}' \in P_{i-1}$, and $\beta \in \Hol(N)$ is an element of order $2$ which commutes with $\alpha_{i-2}$ and the factor $\Hol(C_{p_{i-1}} \rtimes C_{p_i})$ of $\Hol(N)$.

    The statement of the lemma is then the product of these two elements
\end{proof}
\begin{proposition}\label{p_i-groups}
    If $M \leq \Hol(N)$ is  a transitive subgroup such that $\alpha_{i-2} \in \pi(M)$ for some $i \in I$, then $M$ contains one of the following as a Sylow $p_{i-1}$-subgroup:
    \[\langle \alpha_{i-2},\e_1,\e_2 \rangle, \hspace{1cm} \langle \alpha_{i-2},\e_1 \rangle, \hspace{1cm} \langle \alpha_{i-2},\e_2 \rangle,\]
        \begin{align*}
            &\begin{rcases*}
                \langle \e_1,(\alpha_{i-2},\lambda\e_2) \rangle,&\\
                \langle (\alpha_{i-2},\lambda\e_2) \rangle,&
            \end{rcases*}
            &&\text{only if } \omega(M)=\langle S_i \rangle,\\
            &\begin{rcases*}
                \langle (\alpha_{i-2},\lambda\e_1),\e_2 \rangle,&\\
                \langle (\alpha_{i-2},\lambda\e_1) \rangle,&
            \end{rcases*}
            &&\text{only if } \omega(M)=\langle S_iA_i^{-1} \rangle.
        \end{align*}
        where $1 \leq \lambda \leq p_{i-1}-1$.
\end{proposition}
\begin{proof}
    An arbitrary subgroup of $\langle \alpha_{i-2},\e_1,\e_2 \rangle$ has the form
    \begin{equation}\label{p_i-subgroup}
        \langle (\alpha_{i-2}^{\nu_1},\lambda_1\e_1+\mu_1\e_2),(\alpha_{i-2}^{\nu_2},\mu_2\e_2),\alpha_{i-2}^{\nu_3} \rangle.
    \end{equation}
    We may assume that $\nu_1,\nu_2,\nu_3 \in \left\{0,1\right\}$, not all $0$ as $\alpha_{i-2} \in M$. We have that $0 \leq \lambda_1,\mu_1,\mu_2 \leq p_{i-1}-1$, but we assume that $\lambda_1,\mu_1,\mu_2$ are not all $0$ as $M$ is transitive.
	
    First, let $\nu_3=1$; then we may assume that $\nu_1=\nu_2=0$, so (\ref{p_i-subgroup}) is either $\langle \alpha_{i-2},\e_1,\e_2 \rangle$, $\langle \alpha_{i-2},\e_1 \rangle$, $\langle \alpha_{i-2},\e_2 \rangle$, or $\langle \alpha_{i-2},\lambda\e_1+\e_2 \rangle$. In the latter case, by Lemma \ref{coprime_elt}, $M$ contains an element of the form $(\beta,[\mathbf{v},SA_i^aB_i^b])$, where $0 \leq a \leq p_i-1$, $0 \leq b \leq 1$, $\mathbf{v} \in P_{i-1}$, and $\beta \in \Hol(N)$ an element of order $2$ commuting with $\alpha_{i-2}$ and the factor $\Hol(C_{p_{i-1}} \rtimes C_{p_i})$ of $\Hol(N)$. Therefore $M$ must also contain
    \[(\beta,[\mathbf{v},S_iA_i^aB_i^b])(1,[\lambda\e_1+\e_2,1])(\beta,[\mathbf{v},S_iA_i^aB_i^b])^{-1}=(-1)^bk_i^a(k_i\lambda\e_1+\e_2),\]
    which is possible if and only if $\lambda=0$.
	
    Now suppose $\nu_3=0, \nu_2=1$. Thus we may take $\nu_1=0$, so (\ref{p_i-subgroup}) is either $\langle\lambda\e_1+\e_2,(\alpha_{i-2},\mu\e_2)\rangle$, $\langle\e_1,(\alpha_{i-2},\mu\e_2)\rangle$, $\langle(\alpha_{i-2},\mu\e_2)\rangle$, or $\langle\e_2,\alpha_{i-2}\rangle$. For the first group, we may again argue that $\lambda=0$ for similar reasons to above. For the second and third groups, $M$ must contain
    \[(\beta,[\mathbf{v},S_iA_i^aB_i^b])(\alpha_{i-2},\mu\e_2)(\beta,[\mathbf{v},S_iA_i^aB_i^b])^{-1}=(\alpha_{i-2},(-1)^bk_i^a\mu\e_2),\]
    which is possible either if $\mu=0$, or $a=b=0$ (that is $\omega(M)=\langle S_i \rangle$).
	
    Finally, suppose $\nu_1\neq 0$ with $\nu_2=\nu_3=0$. Then (\ref{p_i-subgroup}) is either
    \[\langle (\alpha_{i-2},\lambda\e_1+\mu\e_2)\rangle  \text{ or } \langle (\alpha_{i-2},\lambda\e_1),\e_2\rangle.\]
    For the first group, $M$ contains
    \[(\beta,[\mathbf{v},S_iA_i^aB_i^b])(\alpha_{i-2},\lambda\e_1+\mu\e_2)(\beta,[\mathbf{v},S_iA_i^aB_i^b])^{-1}=(\alpha_{i-2},(-1)^bk_i^a(k_i\lambda\e_1+\mu\e_2)),\]
    which is possible if and only if $\mu=0$, $b=0$ and $a=-1$, or $\lambda=0$ and either $\mu=0$ or $a=b=0$ (that is $\omega(M)=\langle S_i \rangle$). For the second group, $M$ must contain
    \[(\beta,[\mathbf{v},S_iA_i^aB_i^b])(\alpha_{i-2},\lambda\e_1)(\beta,[\mathbf{v},S_iA_i^aB_i^b])^{-1}=(\alpha_{i-2},(-1)^bk_i^a\lambda\e_1),\]
    which is possible if $\lambda=0$ or $b=0$ and $a=-1$ (that is $\omega(M)=\langle S_iA_i^{-1} \rangle$).
\end{proof}
Combining Propositions \ref{p_i-groups} and \ref{2-groups} allows us to find all transitive subgroups $M \leq \Hol(N)$. We now need to understand when two transitive subgroups $M,M' \leq \Hol(N)$ are isomorphic as permutation groups.
\begin{proposition}\label{direct_isoms}
    Suppose that we have transitive subgroups $M_0,M_0' \leq \Hol(C_m)$ and for each $i \in I$, $M_i,M_i' \leq \Hol(C_{p_{i-1}} \rtimes C_{p_i})$ so that
    \begin{align*}
        &M:=M_0 \times \prod_{i \in I}M_i,\\
        &M':=M_0' \times \prod_{i \in I}M_i'
    \end{align*}
    are both transitive subgroups of $\Hol(N)$. Then $M$ and $M'$ are isomorphic as permutation groups if and only if $M_j$ and $M_j'$ are isomorphic as permutation groups for each $j \in I \cup \{0\}$.
\end{proposition}
\begin{proof}
    We first assume that $|I|=1$, so $M=M_0 \times M_1$ and $M'=M_0' \times M_1'$ with $M_1,M_1' \leq \Hol(C_p \rtimes C_q)$ for some primes $p,q$ with $p=2q+1$. It is clear that if $M_0 \cong M_0'$ and $M_1 \cong M_1'$ as permutation groups, then $M \cong M'$ as permutation groups. Suppose then that we have an isomorphism $\phi: M \to M'$ of permutation groups, but that $\phi(M_0) \neq M_0'$ and $\phi(M_1) \neq M_1'$. Then $\phi$ either sends an element $\alpha \in M_0 \cap \Aut(C_m)$ of order $p$ to an element in $\langle \e_1,\e_2 \rangle \subseteq M_1'$ or an element $\beta \in M_0 \cap \Aut(C_m)$ of order $2$ to an element $[\mu\e_j,B]$ of order $2$ in  $M_1'$ for some $i \in \{1,2\}$ and $0 \leq \mu \leq p-1$. In the first case, $\alpha$ acts on an element $\sigma \in C_m$ of order $r:=(p-1)/2$, but neither $\e_1$ nor $\e_2$ act on an element of order $r$ (in fact there is no element of order $r$ in $M_1'$), so $\phi(\alpha)$ must lie in $M_0'$. The second case is similar; indeed $\beta$ acts on an element of order coprime to $p$, but $[\mu\e_j,B]$ acts on an element of order $p$, so $\phi(\beta) \in M_0'$. We conclude that $\phi(M_0) \subseteq M_0'$, and therefore $\phi(M_0) = M_0'$, meaning we must also have $\phi(M_1)=M_1'$.

    Finally, we note that this argument can be extended for $|I|>1$. For example, let $M=M_0 \times M_1 \times M_2$ and $M'=M_0' \times M_1' \times M_2'$. If there is an isomorphism $\phi:M \to M'$, then $\phi(M_0) \cap M_1'=\phi(M_0) \cap M_2' = \{1\}$ by the same logic as above (and thus $\phi(M_0)=M_0'$). Further, $\gcd(|M_1|,|M_2'|)=\gcd(|M_1'|,|M_2|) \leq 2$, but an element of order $2$ in $M_1$ acts on an element of order, say, $p$, and an element of order $2$ in $M_2'$ acts on an element of order coprime with $p$. We must therefore have $\phi(M_1)=M_1'$ and $\phi(M_2)=M_2'$.
\end{proof}
We remark that the argument used in Proposition \ref{direct_isoms} also shows that if $M,M' \leq \Hol(N)$ are transitive subgroups such that $M$ is a direct product of transitive subgroups as described in the proposition, but $M'$ contains elements of the form $(\alpha,\mu\e_j)$ or $(\beta,[\mu\e_j,B])$, then $M$ cannot be isomorphic to $M'$, even as abstract groups. Therefore, we are left to determine when two transitive subgroups of the form described for $M'$ are isomorphic as permutation groups.
\begin{proposition}\label{indirect_isoms}
    Let $M_1,M_2 \leq \Hol(N)$ be transitive subgroups differing only by the fact that $M_1$ contains the generator $(\alpha_{i-2}^{\nu},\e_j)$ and $M_2$ instead contains the generator $(\alpha_{i-2},\e_j)^{\nu}$ for some $j \in \{1,2\}$ and $1 \leq \nu \leq p_{i-1}-1$. Define the map $\varphi_1:M_1 \to M_2$ by
    \[\varphi_1((\alpha_{i-2}^{\nu},\e_j))=(\alpha_{i-2},\e_j)^{\nu}\]
    and acting as identity on the other generators. Then $\varphi_1$ gives an isomorphism of permutation groups.\\

    Let $M_3,M_4 \leq \Hol(N)$ be transitive subgroups differing only by the fact that $M_3$ contains the generators $(\alpha_{i-2}^{\nu},\e_1)$, $[\mu\e_2,S_iA_i^{-1}]$ and $\e_2$, and $M_4$ instead contains the generators $(\alpha_{i-2},\e_2)^{\nu}$, $[\mu\e_1,S_i]$ and $\e_1$, for some $1 \leq \nu \leq p_{i-1}-1$ and $0 \leq \mu \leq p_{i-1}-1$. Define the map $\varphi_2:M_3 \to M_4$ by
    \[\varphi_2((\alpha_{i-2}^{\nu},\e_1))=(\alpha_{i-2},\e_2)^{\nu}, \;\; \varphi_2([\mu\e_2,S_iA_i^{-1}])=[\mu\e_1,S_i], \;\; \varphi_2(\e_2)=\e_1,\]
     and acting as identity on the other generators. Then $\varphi_2$ (and hence also $\varphi_2^{-1}$) gives an isomorphism of permutation groups.\\
    
    Other than the isomorphisms already described, there are no further isomorphisms between the transitive subgroups of $\Hol(N)$, even as abstract groups.
\end{proposition}
\begin{proof}
    It is clear that the maps $\varphi_1$ and $\varphi_2$ give isomorphisms of permutation groups which restrict to isomorphisms of any two subgroups of the form $\langle (\alpha_{i-2}^{\nu},\e_j) \rangle$. By Corollary \ref{2-isoms}, without loss of generality, we now only need to show that if $M,M' \leq \Hol(N)$ are transitive subgroups such that $M$ and $M'$ contain $2$-subgroups of $B$, as given by (\ref{B}), of the same order, but there is a $2$-subgroup $Y \leq M \cap H$ such that $Y \not\subseteq M'$, then $M$ and $M'$ are not isomorphic as permutation groups. To this end, assuming $M \cong M'$, we observe that $Y \leq \Aut(N)$ with $Y$ abelian and both $M$ and $M'$ must contain a characteristic abelian subgroup $N_I \leq C_m$. We may then apply the method used in the proof of Proposition \ref{cyclic_isoms_cunningham} to show that $M'$ must contain $Y$, which contradicts our assumption and finishes the proof.
\end{proof}
The next proposition gives us a way of computing $|\Aut(M,M')|$ for each transitive subgroup $M \leq \Hol(N)$.
\begin{proposition}\label{metab_cunningham_isoms}
    Let $M \leq Hol(N)$ be a transitive subgroup, and let $M_X$ denote the projection of $M$ to the factor $X$ of $\Hol(N)$, where $X$ is one of $\Hol(C_{m_j})$ (with each $m_j$ a Cunningham product such that $\prod m_j =m$) or $\Hol(C_{p_{i-1}} \rtimes C_{p_i})$ for some $i \in I$. Let $\mathcal{X}:=\{X \mid X \text{ is a factor of }\Hol(N)\}$. Then
    \[|\Aut(M,M')|=\prod_{X \in \mathcal{X}}|\Aut(M_X,M_X')|.\]
\end{proposition}
\begin{proof}
    Let $\phi:M \to M$ be an isomorphism of permutation groups. Then $\phi$ is a composition of isomorphisms of each factor $X \in \mathcal{X}$ and isomorphisms of the form $\varphi_1$ or $\varphi_2$ described in Proposition \ref{indirect_isoms}. However, we note that $\varphi_1(M),\varphi_2(M) \neq M$, so $\phi$ must just be composed of isomorphisms on each factor.
\end{proof}
We now summarise the results of this section, as a tool for computing and counting the Hopf--Galois structures of type $N$.

How the transitive subgroups of $\Hol(N)$ are constructed:
\begin{enumerate}[label=$\bullet$]
    \item Remark \ref{cunningham_metacyclic_observations} tells us that a transitive subgroup of $\Hol(N)$ must also restrict to a transitive subgroup of $\Hol(C_m)$ and a transitive subgroup of $\Hol(C_{p_{i-1}} \rtimes C_{p_i})$ for each $i \in I$.
    \item Section \ref{cyclic_cunningham} tells us what the transitive subgroups of $\Hol(C_m)$ look like (note that we may have to split into several factors $\Hol(C_{m_j})$ so that each $m_j$ is a Cunningham product).
    \item Section 4.2 of \cite{BML21} tells us what the transitive subgroups of $\Hol(C_{p_{i-1}} \rtimes C_{p_i})$ look like.
    \item Propositions \ref{2-groups} and \ref{p_i-groups} tell us how to glue the groups above to create all transitive subgroups $M$ of $\Hol(N)$.
\end{enumerate}
How to sort the transitive subgroups into isomorphism classes and count the numbers of Hopf--Galois structures admitted:
\begin{enumerate}[label=$\bullet$]
    \item Propositions \ref{direct_isoms} and \ref{indirect_isoms} tell us when two transitive subgroups $M_1,M_2 \leq \Hol(N)$ are isomorphic as permutation groups, given the isomorphisms on each factor of $\Hol(N)$. To explicitly understand the isomorphisms on the $\Hol(C_m)$ component, we refer to Section \ref{cyclic_cunningham}. We refer to Section 4.2 of \cite{BML21} to understand the isomorphisms on the factors $\Hol(C_{p_{i-1}} \rtimes C_{p_i})$.
    \item Proposition \ref{metab_cunningham_isoms} tells us how to compute $|\Aut(M,M')|$ given how $M$ restricts to each component of $\Hol(N)$, using Section \ref{cyclic_cunningham} and Table \ref{metab-trans-HGS}, which can also be found in Section 4.2 of \cite{BML21}.
    \item Proposition \ref{cunningham_2gp_count} as well as Section \ref{cyclic_cunningham} and Section 4.2 of \cite{BML21} help towards counting how many groups are in the same isomorphism class.
    \item The above results are combined with Lemma \ref{Byott_num_HGS} to obtain the number of Hopf--Galois structures of type $N$ with transitive permutation group $M$.
\end{enumerate}
We finish this paper by illustrating the above with some examples.

\section{Examples}
In this section, we apply the results of the previous section to separable extensions of degree $n$ where $n$ is a Cunningham product of length three or four. We fully analyse the situation $n=p_1p_2p_3$ and the cyclic case of $n=p_1p_2p_3p_4$. The four metacyclic cases in the length four setting follow very similarly to length three setting.
\begin{remark}
    Of course the number of types of Hopf--Galois structures grows proportionally to $\phi^l$, where $\phi=(1+\sqrt{5})/2$ and $l$ is the length of the associated Cunningham chain; that is it follows the growth of the Fibonacci sequence.
    
    A general comment on the number of (and even isomorphism classes of) transitive subgroups of $\Hol(N)$ for some $N$ of size $n=p_1\cdots p_l$ is markedly difficult to give. This is because the count is heavily dependent on the number of metacyclic factors in the decomposition of $N$, and where they are placed. Where there are no metacyclic factors, Theorem \ref{cyclic_cunningham} gives an explicit count of isomorphism classes, but here the formula is already rather complicated (although again it is clear that we have some dependence on $F(l)$, that is, the number of choices for $I$). Where (non-trivial) metacyclic factors are introduced, the results displayed in Theorems \ref{N2_tot} and \ref{N3_tot} demonstrate that even the position of a metacyclic factor can greatly affect the count. This is because it matters which factor has order divisible by $p_l$, and whether there is an odd prime dividing both $\Hol(A)$ and $\Hol(B)$ for some $A,B \leq N$, where $A$ is a cyclic factor of $N$ and $B$ is a metacyclic factor of $N$.
\end{remark}
We start by including two tables of results from Section 4.2 of \cite{BML21} which are used in some of the computations below. These explicitly describe and summarise Section \ref{metab-cunningham} for the case $n=p_1p_2$. For ease of notation, we let $p:=p_1$, $q:=p_2$, $P:=P_1:=\langle \e_1,\e_2 \rangle$, $R:=R_2=\langle S,A,B \rangle$, and $k:=k_2$. We note that the expression $\mathbb{F}_p^2 \rtimes_u C_q$ denotes a unique isomorphism class for each $u$, as discussed in \cite{BML21}.
\begin{table}
\centering
    \begin{tabular}{|c|c|c|c|c|c|c|c|} \hline
        Order & Parameters  & Structure & Group\\
        \hline 
        $p^2 q^2$ & & $\mathbb{F}_p^2 \rtimes (C_q \times C_q)$ &  $P \rtimes \langle S, A\rangle$\\
        \hline
        $2 p^2 q^2 $ & & $\mathbb{F}_p^2 \rtimes(C_q \times C_{2q})$   & $\Hol(N)$\\
        \hline
        $p^2 q$ & $0 \leq u \leq q-1$ & $\mathbb{F}_p^2 \rtimes C_q$ & $P \rtimes\langle S A^u \rangle$\\
        \hline
        $2p^2 q$ & $0 \leq u \leq q-1$ & $\mathbb{F}_p^2 \rtimes C_{2q}$  & $P \rtimes \langle SA^u, B \rangle$\\
        \hline
        $pq^2$ & \multirow{2}{3em}{$\lambda \in \mathbb{F}_p$} & \multirow{2}{7em}{$C_q\times(C_p \rtimes C_q)$} & $\langle \e_1,S, [\lambda \e_2, A] \rangle$\\ 
        &  & &
       $\langle \e_2,  [\lambda \e_1,S], [\lambda \e_1, A] \rangle$\\
       \hline
        $2pq^2$ & \multirow{2}{3em}{$\lambda \in \mathbb{F}_p$} & \multirow{2}{7em}{$C_q \times (C_p \rtimes C_{2q})$ }& $\langle \e_1,S, [\lambda (1-k)\e_2, A], [2 \lambda \e_2, B] \rangle$\\   
        &  & & $\langle \e_2, [\lambda\e_1,S], [\lambda (1-k)\e_1, A], [2\lambda \e_1, B] \rangle$\\
        \hline
        $pq$ &  & $C_p\rtimes C_q$ & $\langle \e_1, [\lambda \e_2, SA^u] \rangle$\\
        &  & $C_{pq}$ & $\langle \e_1,
        [\lambda \e_2, SA^{-1}] \rangle$\\ 
        & $1 \leq u \leq q-2$, & $C_p \rtimes C_q$ & $\langle \e_1, S \rangle = N$\\ 
        & $\lambda \in \mathbb{F}_p$ & $C_p
        \rtimes C_q$ &  $\langle \e_2, [\lambda \e_1, SA^u]\rangle$\\ 
        & & $C_{pq}$ & $\langle \e_2,[\lambda \e_1, S] \rangle$\\ 
        & & $C_p \rtimes C_q$ &  $\langle \e_2, SA^{-1} \rangle$ \\ \hline
        $2pq$ &  &  $C_p\rtimes C_{2q}$ & $\langle \e_1, [\lambda (1-k^u)\e_2, SA^u],[2 \lambda\e_2, B] \rangle$ \\ 
        &  & $D_{2p} \times C_q$ & $\langle \e_1, [\lambda(1-k^{-1})\e_2, SA^{-1}], [2\lambda \e_2, B]\rangle$\\ 
        & $1 \leq u \leq q-2$, & $C_p \rtimes C_{2q}$ & $\langle \e_1, S,[\lambda \e_2, B] \rangle$\\ 
        & $\lambda \in \mathbb{F}_p$ &  $C_p \rtimes C_{2q}$  & $\langle \e_2, [\lambda (1-k^{u+1})\e_1, SA^u], [2\lambda \e_1, B] \rangle$\\           
        &  &  $D_{2p} \times C_q$ & $\langle \e_2, [\lambda(1-k)\e_1, S], [2\lambda \e_1,B] \rangle$\\                     
        &  &  $C_p \rtimes C_{2q}$ & $\langle \e_1, SA^{-1}, [\lambda \e_2, B] \rangle$\\ 
        \hline
    \end{tabular}
\vskip3mm
\caption{Transitive subgroups of $\Hol(C_p \rtimes C_q)$.} 
 \label{metab-trans-subgroups}  	
\end{table}
\begin{table}
\centering
\begin{tabular}{|c|c|c|c|c|} \hline
  Order & Structure & $\#$ groups & $|\Aut(M,M')|$ & $\#$ HGS \\ \hline
  $p^2 q^2$ & $N \rtimes (C_p \rtimes C_q)$ & $1$ & $2p(p-1)$ & $2$\\
  \hline
  $2p^2 q^2$ & $\Hol(N)$ & $1$ & $2p(p-1)$ & $2$ \\
  \hline
  $p^2 q$ & $C_p \times (C_p \rtimes C_q)$  & $2$ & $p(p-1)$ & $2p$ \\
  & $\mathbb{F}_p^2 \rtimes_u C_q$, $1 \leq u \leq \frac{1}{2}(q-3)$ & $2$ & $p^2(p-1)$ & $2p$ \\
   & $\mathbb{F}_p^2 \rtimes_{\frac{1}{2}(q-1)} C_q$ & $1$ & $2p^2(p-1)$ & $2p$ \\ 
   \hline
  $2p^2 q$ & $(C_p \times (C_p \rtimes C_q))\rtimes C_2$ & $2$ & $p^2(p-1)$ & $2p$ \\
   & $\mathbb{F}_p^2 \rtimes_u C_{2q}$, $1 \leq u \leq \frac{1}{2}(q-3)$ & $2$ & $p^2(p-1)$ & $2p$ \\
  & $\mathbb{F}_p^2 \rtimes_{\frac{1}{2}(q-1)} C_{2q}$ & $1$ & $2p^2(p-1)$ & $2p$ \\ \hline
  $pq^2$ & $C_q \times (C_p \rtimes C_q)$ & $2p$ & $(p-1)(q-1)$ & $2(q-1)$ \\ \hline
  $2pq^2$ & $C_q \times (C_p \rtimes C_{2q})$ & $2p$ & $(p-1)(q-1)$ & $2(q-1)$ \\ \hline
  $pq$ &  $C_p \rtimes C_q$ & $2p(q-2)+2$ & $p(p-1)$ & $2p(q-2)+2$ \\
   & $C_{pq}$ & $2p$ & $(p-1)(q-1)$ & $2(q-1)$ \\ \hline
  $2pq$ &  $C_p \rtimes C_{2q}$ & $2p(q-1)$ & $p-1$ & $2(q-1)$ \\
   &  $D_{2p} \times C_q$ & $2p$ & $(p-1)(q-1)$ & $2(q-1)$ \\
\hline
\end{tabular}
\vskip3mm
\caption{\#HGS of type $C_p \rtimes C_q$.} 
 \label{metab-trans-HGS}  	
\end{table}

\subsection{\texorpdfstring{$n=p_1p_2p_3$}{n=p1q2r3}}
There are $F(3)=3$ groups of order $n=p_1p_2p_3$ where $p_1=2p_2+1,p_2=2p_3+1$. Let $p_3-1=2^xs$ with $s$ odd. As mentioned in Example \ref{fib_3_example}, these are the cyclic group, $N_1:=C_n$ and the two metacyclic groups $N_2:=(C_{p_1} \rtimes C_{p_2}) \times C_{p_3}$ and $N_3:=C_{p_1} \times (C_{p_2} \rtimes C_{p_3})$.

We first look at the number of Hopf--Galois structures of type $N_1$. We have $l=3$, and so our choices for $I$ are $\emptyset,\{2\}$, and $\{3\}$. We have:
\begin{align*}
    &J_{\emptyset,\emptyset}=N_1,\\
    &J_{\{2\},t_2}=\left\langle \sigma_1, \left[\sigma_2,\alpha_1^{t_2}\right], \sigma_3\right\rangle \cong N_2,\\
    &J_{\{3\},t_3}=\left\langle \sigma_1, \sigma_2, \left[\sigma_3,\alpha_2^{t_3}\right] \right\rangle \cong N_3,
\end{align*}
where $1 \leq t_i \leq p_i-1$ for each $t_i$ above. Therefore, Proposition \ref{cyclic_trans} tells us that we have the following list of transitive subgroups of $\Hol(N_1)$:
\begin{align*}
    &N_1 \rtimes A,                  &&A \leq \Aut(N_1),\\
    &J_{\{2\},t_2} \rtimes B,    &&B \leq \langle \beta_1,\alpha_3,\beta_3 \rangle,\\
    &J_{\{3\},t_3} \rtimes C,    &&C \leq \langle \alpha_1,\beta_1,\beta_2 \rangle.
\end{align*}
Proposition \ref{cyclic_HGS_Cunningham} tells us that these correspond to extensions admitting either a unique Hopf--Galois structure or $p_2$ or $p_3$ Hopf--Galois structures, depending on the choices of $B$ and $C$. Theorem \ref{cyclic_cunningham_thm} then tells us that there are $(47x+22)\sigma_0(s)+10$ isomorphism types of permutation groups of degree $p_1p_2p_3$ which are realised by a Hopf--Galois structure of type $N_1$.

For the remainder of this example, we let $p:=p_1, q:=p_2$ and $r:=p_3$. 

We now treat the case for the Hopf--Galois structures of type $N_2$. We have \[N_2= (C_p \rtimes C_q) \times C_r.\]
For ease of notation, we let $P:=P_1=\langle \e_1,\e_2 \rangle$, $R:=R_2:=\langle S,A,B \rangle$, and $k:=k_2$. We further set $\Aut(C_r)=\langle \gamma, \delta \rangle$. Therefore
\[\Hol(N_2)=\langle \e_1,\e_2 \rangle \rtimes \langle S,A,B \rangle \times (\langle \sigma \rangle \rtimes \langle \gamma, \delta \rangle).\]
Table \ref{N2_trans_groups} lists the transitive subgroups of $\Hol(N_2)$, Table \ref{N2_isom_classes} gives the isomorphism types of each subgroup, and Table \ref{N2_HGS} gives the number of Hopf--Galois structures per isomorphism class.
\begin{table}
    \centering
    \scalebox{0.97}{
    \hspace*{-1.5cm}\begin{tabular}{|c|c|}
    \hline
    Parameters  &  Group\\
    \hline
    $a \mid s,b\in \{0,1\}$ & $(P \rtimes \langle S,A,B^b \rangle)\times (\langle \sigma \rangle \rtimes X)$\\
    $1 \leq c \leq x$ & $\langle \e_1,\e_2,S,A,(\gamma^{2^{x-c}},B), \sigma, \delta^{s/a} \rangle$\\
    \hline 
    $b \in \{0,1\}$, $0 \leq u \leq q-1$, & $(P \rtimes \langle S A^u, B^b \rangle) \times  (\langle \sigma \rangle \rtimes X)$\\
    $1 \leq c \leq x$ & $\langle \e_1,\e_2, SA^u, (\gamma^{2^{x-c}},B), \sigma, \delta^{s/a}, \rangle$\\
    \hline
    & $(\langle \e_1, S, [(1-k)\lambda\e_2, A], [(1-(-1)^b)\lambda\e_2, B^b] \rangle) \times  (\langle \sigma \rangle \rtimes X)$\\
    $0 \leq \lambda \leq p-1$, $1 \leq c \leq x$ & $\langle \e_1, S, [(1-k)\lambda\e_2, A], (\gamma^{2^{x-c}},[2\lambda\e_2, B]), \sigma, \delta^{s/a} \rangle$\\
    $a \mid s,b \in \{0,1\}$ & $\langle \e_2, [(1-k)\lambda\e_1,S], [(1-k)\lambda\e_1, A], [(1-(-1)^b)\lambda\e_1, B^b] \rangle \times  (\langle \sigma \rangle \rtimes X)$\\
    & $\langle \e_2, [(1-k)\lambda\e_1,S], [(1-k)\lambda\e_1, A], (\gamma^{2^{x-c}},[2\lambda\e_1, B]), \sigma, \delta^{s/a} \rangle$\\
    \hline
    & $(\langle \e_1, [(1-k^u)\lambda\e_2, SA^u], [(1-(-1)^b)\lambda\e_2, B^b] \rangle) \times  (\langle \sigma \rangle \rtimes X)$\\
    $0 \leq u \leq q-1$, $0 \leq \lambda \leq p-1$, & $\langle \e_1, [(1-k^u)\lambda\e_2, SA^u], (\gamma^{2^{x-c}},[2\lambda\e_2, B]), \sigma, \delta^{s/a} \rangle$\\
    $a \mid s, b \in \{0,1\}, 1 \leq c \leq x$ & $(\langle \e_2, [\lambda (1-k^{u+1})\e_1, SA^u], [(1-(-1)^b)\lambda\e_1, B^b] \rangle) \times  (\langle \sigma \rangle \rtimes X)$\\
    & $\langle \e_2, [(1-k^{u+1})\lambda\e_1, SA^u], (\gamma^{2^{x-c}},[2\lambda\e_1, B]), \sigma, \delta^{s/a}, \rangle$\\
\hline
\end{tabular}
}
    \vspace{5mm}
    \caption{Transitive subgroups of $\Hol(N_2)$}
    \label{N2_trans_groups}
\end{table}
\begin{remark}\label{table_data}
    To obtain Table \ref{N2_trans_groups} we apply Propositions \ref{p_i-groups} and \ref{2-groups} to the transitive subgroups of metacyclic type given in Table \ref{metab-trans-subgroups}, and the transitive subgroup of the holomorph of a group of prime order.

    To obtain Table \ref{N2_HGS}, we apply Proposition \ref{metab_cunningham_isoms} along with the data in Table \ref{metab-trans-HGS}.

    We let $X$ run through the subgroups of $\langle \gamma, \delta \rangle$.
\end{remark}
We obtain the following theorem by counting the number of representatives in Table \ref{N2_isom_classes}.
\begin{theorem}\label{N2_tot}
    In total, there are $49+5r$ isomorphism types of permutation groups of degree $pqr$ which are realised by a Hopf--Galois structure of type $N_2$.
\end{theorem}
\begin{sidewaystable}
\setlength{\extrarowheight}{1.5mm}
	\vskip135mm
	\bigskip
	\centering
        \scalebox{1.25}{
    \begin{tabular}{|c|c|c|}
    \hline
    Parameters & Class Representative & Structure\\
    \hline
    $b \in \{0,1\}$  & $(P \rtimes \langle S,A,B^b\rangle) \times (\langle \sigma \rangle \rtimes X)$ & $((C_p \times (C_p \rtimes C_q)) \rtimes C_{bq}) \times (C_r \rtimes C_d)$\\
    \hline
    $a \mid s$, $1 \leq c \leq x$ & $\langle \e_1,\e_2, S, A, (\gamma^{2^{x-c}},B), \sigma, \delta^{s/a} \rangle$ & $(((C_p \times (C_p \rtimes C_q))\rtimes C_q) \times (C_r \rtimes C_a)) \rtimes C_{2^c}$\\
    \hline
    $b \in \{0,1\}$, $1 \leq u \leq \frac{1}{2}(q-3)$ & $(P \rtimes \langle S A^u, B^b \rangle) \times (\langle \sigma \rangle \rtimes X)$ & $(\mathbb{F}_p^2 \rtimes_u C_{bq}) \times (C_r \rtimes C_d)$\\
    \hline
    $b \in \{0,1\}$ & $(P \rtimes \langle S A^{-1}, B^b \rangle) \times (\langle \sigma \rangle \rtimes X)$ & $((C_p \times (C_p \rtimes C_q)) \rtimes C_b) \times (C_q \rtimes C_d)$\\
    \hline
    $b \in \{0,1\}$ & $(P \rtimes \langle S A^{\frac{1}{2}(q-1)}, B^b \rangle) \times (\langle \sigma \rangle \rtimes X)$ & $((\mathbb{F}_p^2 \rtimes_{\frac{1}{2}(q-1)} C_q) \rtimes C_b) \times (C_r \rtimes C_d)$\\
    \hline
    $a \mid s$, $1 \leq u \leq \frac{1}{2}(q-3)$, & $\langle \e_1,\e_2, SA^u, (\gamma^{2^{x-c}},B), \sigma, \delta^{s/a} \rangle$ & $((\mathbb{F}_p^2 \rtimes_u C_q) \times (C_r \rtimes C_a)) \rtimes C_{2^c}$\\
    $1 \leq c \leq x$ & &\\
    \hline
    $a \mid s$, $1 \leq c \leq x$ & $\langle \e_1,\e_2, SA^{-1}, (\gamma^{2^{x-c}},B), \sigma, \delta^{s/a} \rangle$ & $((C_p \times (C_p \rtimes C_q)) \times (C_r \rtimes C_a)) \rtimes C_{2^c}$\\
    \hline
    $a \mid s$, $1 \leq c \leq x$ & $\langle \e_1,\e_2, SA^{\frac{1}{2}(q-1)}, (\gamma^{2^{x-c}},B), \sigma, \delta^{s/a} \rangle$ & $((\mathbb{F}_p^2 \rtimes_{\frac{1}{2}(q-1)} C_q) \times (C_r \rtimes C_a)) \rtimes C_{2^c}$\\
    \hline
    $b \in \{0,1\}$ & $\langle \e_1,S,A,B^b \rangle \times (\langle \sigma \rangle \rtimes X)$ & $(C_q \times (C_p \rtimes C_{bq})) \times (C_r \rtimes C_d)$\\
    \hline
    $a \mid s$, $1 \leq c \leq x$ & $\langle \e_1,S,A,(\gamma^{2^{x-c}},B), \sigma,\delta^{s/a} \rangle$ & $((C_q \times (C_p \rtimes C_q)) \times (C_r \rtimes C_d)) \rtimes C_{2^x}$\\
    \hline
    $b \in \{0,1\}$, $1 \leq u \leq q-2$ & $\langle \e_1,SA^u,B^b \rangle \times (\langle \sigma \rangle \rtimes X) $ & $(C_p \rtimes C_{bq}) \times (C_r \rtimes C_d)$ \\
    \hline
    $b \in \{0,1\}$ & $\langle \e_1,SA^{-1},B^b \rangle \times (\langle \sigma \rangle \rtimes X)$ & $((C_p \rtimes C_b) \times C_q) \times (C_r \rtimes C_d)$\\
    \hline
    $a \mid s$, $1 \leq u \leq q-2$, & $\langle \e_1,SA^u,(\gamma^{2^{x-c}},B),\sigma,\delta^{s/a} \rangle$ & $((C_p \rtimes C_q) \times (C_r \rtimes C_a)) \rtimes C_{2^x}$\\
    $1 \leq c \leq x$ & &\\
    \hline
    $a \mid s$, $1 \leq c \leq x$ & $\langle \e_1,SA^{-1},(\gamma^{2^{x-c}},B), \sigma,\delta^{s/a} \rangle$ & $((C_p \times C_q) \times (C_r \rtimes C_a)) \rtimes C_{2^x}$\\
    \hline
    \end{tabular}}
    \vskip5mm
    \caption{Isomorphism classes for the transitive subgroups of $\Hol(N_2)$}
    \label{N2_isom_classes}
\end{sidewaystable}
\begin{sidewaystable}
    \setlength{\extrarowheight}{1.5mm}
    \vskip135mm
    \bigskip
    \centering
    \scalebox{1.25}{
        \begin{tabular}{|c|c|c|c|}
            \hline
            Structure & \# groups & $|\Aut(M,M')|$ & \# HGS\\
            \hline
            $((C_p \times (C_p \rtimes C_q)) \rtimes C_{bq}) \times (C_r \rtimes C_d)$ & $1$ & $2p(p-1)(r-1)$ & $2$\\
            \hline
            $(((C_p \times (C_p \rtimes C_q))\rtimes C_q) \times (C_r \rtimes C_a)) \rtimes C_{2^c}$ & $1$ & $2p(p-1)(r-1)$ & $2$\\
            \hline
            $(\mathbb{F}_p^2 \rtimes_u C_{bq}) \times (C_r \rtimes C_d)$ & $2$ & $p^{3-b}(p-1)(r-1)$ & $2p^{2-b}$\\
            \hline
            $((C_p \times (C_p \rtimes C_q)) \rtimes C_b) \times (C_q \rtimes C_d)$ & $2$ & $p(p-1)(r-1)$ & $2$\\
            \hline
            $((\mathbb{F}_p^2 \rtimes_{\frac{1}{2}(q-1)} C_q) \rtimes C_b) \times (C_r \rtimes C_d)$ & $1$ & $2p^2(p-1)(r-1)$ & $2p$\\
            \hline
            $((\mathbb{F}_p^2 \rtimes_u C_q) \times (C_r \rtimes C_a)) \rtimes C_{2^c}$ & $2$ & $p(p-1)(r-1)$ & $2$\\
            \hline
            $((C_p \times (C_p \rtimes C_q)) \times (C_r \rtimes C_a)) \rtimes C_{2^c}$ & $2$ & $p(p-1)(r-1)$ & $2$\\
            \hline
            $((\mathbb{F}_p^2 \rtimes_{\frac{1}{2}(q-1)} C_q) \times (C_r \rtimes C_a)) \rtimes C_{2^c}$ & $1$ & $2p^2(p-1)(r-1)$ & $2p$\\
            \hline
            $(C_q \times (C_p \rtimes C_{bq})) \times (C_r \rtimes C_d)$ & $2p$ & $(p-1)(q-1)(r-1)$ & $2(q-1)$\\
            \hline
            $((C_q \times (C_p \rtimes C_q)) \times (C_r \rtimes C_d)) \rtimes C_{2^x}$ & $2p$ & $(p-1)(q-1)(r-1)$ & $2(q-1)$\\
            \hline
            $(C_p \rtimes C_{2q}) \times (C_r \rtimes C_d)$ & $2p(q--1)$ & $(p-1)(r-1)$ & $2(q-1)$\\
            $(C_p \rtimes C_q) \times (C_r \rtimes C_d)$ & $2p(q-2)+2$ & $p(p-1)(r-1)$ & $2p(q-2)+2$\\
            \hline
            $((C_p \rtimes C_b) \times C_q) \times (C_r \rtimes C_d)$ & $2p$ & $(p-1)(q-1)(r-1)$ & $2(q-1)$\\
            \hline
            $((C_p \rtimes C_q) \times (C_r \rtimes C_a)) \rtimes C_{2^x}$ & $2p(q-1)$ & $(p-1)(r-1)$ & $2(q-1)$\\
            \hline
            $((C_p \times C_q) \times (C_r \rtimes C_a)) \rtimes C_{2^x}$ & $2p$ & $(p-1)(q-1)(r-1)$ & $2(q-1)$\\
            \hline
        \end{tabular}
        }
        \vspace{5mm}
        \caption{Hopf--Galois structures of type $N_2$}
        \label{N2_HGS}
\end{sidewaystable}
\newpage
Finally, we look at the Hopf--Galois structures of type $N_3$. We have
\[N_3= C_p\times (C_q \rtimes C_r).\]
Similar to the case for $N_2$, for ease of notation, we let $P:=P_2=\langle \e_1,\e_2 \rangle$ and $R:=R_3:=\langle S,A,B \rangle$, and $k:=k_3$. We further set $\Aut(C_p)=\langle \alpha,\beta \rangle$. Therefore
\[\Hol(N_3)=(\langle \sigma \rangle \rtimes \langle \alpha,\beta \rangle) \times (\langle \e_1,\e_2 \rangle \rtimes \langle S,A,B \rangle).\]
Table \ref{N3_trans_groups} lists the transitive subgroups of $\Hol(N_3)$, Table \ref{N3_isom_classes} gives the isomorphism types of each subgroup, and Table \ref{N3_HGS} gives the number of Hopf--Galois structures per isomorphism class. The data in these tables are obtained in a similar way to Tables \ref{N2_trans_groups}, \ref{N2_isom_classes} and \ref{N2_HGS}. We let $X$ run through the subgroups of $\langle \alpha,\beta \rangle$. 
\begin{table}
    \centering
    \scalebox{0.97}{
    \hspace*{-1.5cm}\begin{tabular}{|c|c|}
    \hline
    Parameters  &  Group\\
    \hline
    $a,b\in \{0,1\}$ & $(\langle \sigma \rangle \rtimes X) \times (P \rtimes \langle S,A,B^b \rangle)$\\
    & $\langle \sigma, \alpha^a, \e_1,\e_2,S,A,(\beta,B) \rangle$\\
    \hline 
    $0 \leq u \leq r-1$, & $(\langle \sigma \rangle \rtimes X) \times (P \rtimes \langle S A^u, B^b \rangle$)\\
    $b \in \{0,1\}$ & $\langle \sigma, \alpha^a, \e_1,\e_2, SA^u, (\beta,B) \rangle$\\
    \hline
    & $(\langle \sigma \rangle \rtimes X) \times (\langle \e_1, S, [(1-k)\lambda\e_2, A], [(1-(-1)^b)\lambda\e_2, B^b] \rangle)$\\
    $0 \leq \lambda \leq q-1$, & $\langle \sigma, \alpha^a, \e_1, S, [(1-k)\lambda\e_2, A], (\beta,[2\lambda\e_2, B]) \rangle$\\
    $a,b \in \{0,1\}$ & $(\langle \sigma \rangle \rtimes X) \times (\langle \e_2, [(1-k)\lambda\e_1,S], [(1-k)\lambda\e_1, A], [(1-(-1)^b)\lambda\e_1, B^b] \rangle$\\
    & $\langle \sigma, \alpha^a, \e_2, [(1-k)\lambda\e_1,S], [(1-k)\lambda\e_1, A], (\beta,[2\lambda\e_1, B]) \rangle$\\
    \hline
    & $\langle \sigma, \beta^b, (\alpha,\mu\e_1),[\lambda \e_2, SA^{-1}] \rangle$\\
    $1 \leq \mu \leq q-1$, $0 \leq \lambda \leq q-1$,& $\langle \sigma, \beta^b, (\alpha,\mu\e_1), \e_2, SA^{-1} \rangle$\\
    $b\in \{0,1\}$ & $\langle \sigma, \beta^b, (\alpha,\mu\e_2),[\lambda \e_1, S] \rangle$\\
    & $\langle \sigma, \beta^b, (\alpha,\mu\e_2), \e_1, S\rangle$\\
    \hline
    & $(\langle \sigma \rangle \rtimes X) \times (\langle \e_1, [(1-k^u)\lambda\e_2, SA^u], [(1-(-1)^b)\lambda\e_2, B^b] \rangle)$\\
    $0 \leq u \leq r-1$, $0 \leq \lambda \leq q-1$, & $\langle \sigma, \alpha^a, \e_1, [(1-k^u)\lambda\e_2, SA^u], (\beta,[2\lambda\e_2, B]) \rangle$\\
    $a,b\in \{0,1\}$ & $(\langle \sigma \rangle \rtimes X) \times (\langle \e_2, [\lambda (1-k^{u+1})\e_1, SA^u], [(1-(-1)^b)\lambda\e_1, B^b] \rangle)$\\
    & $\langle \sigma, \alpha^a, \e_2, [(1-k^{u+1})\lambda\e_1, SA^u], (\beta,[2\lambda\e_1, B]) \rangle$\\
\hline
\end{tabular}
}
    \vspace{5mm}
    \caption{Transitive subgroups of $\Hol(N_3)$}
    \label{N3_trans_groups}
\end{table}
\begin{sidewaystable}
\setlength{\extrarowheight}{1.5mm}
	\vskip135mm
	\bigskip
	\centering
	\scalebox{1.25}{
    \begin{tabular}{|c|c|c|}
    \hline
    Parameters & Class Representative & Structure\\
    \hline
    $b \in \{0,1\}$  & $(\langle \sigma \rangle \rtimes X) \times (P \rtimes \langle S,A,B^b\rangle)$ & $(C_p \rtimes C_d) \times ((C_q \times (C_q \rtimes C_r)) \rtimes C_{br})$\\
    \hline
    $a \in \{0,1\}$ &$\langle \sigma, \alpha^a, \e_1,\e_2, S, A, (\beta,B) \rangle$ & $((C_p \rtimes C_{q^a}) \times ((C_q \times (C_q \rtimes C_r)) \rtimes C_r)) \rtimes C_2$\\
    \hline
    $b \in \{0,1\}$, $1 \leq u \leq \frac{1}{2}(r-3)$ &$(\langle \sigma \rangle \rtimes X) \times (P \rtimes \langle S A^u, B^b \rangle)$ & $(C_p \rtimes C_d) \times (\mathbb{F}_q^2 \rtimes_u C_{br})$\\
    \hline
    $b \in \{0,1\}$ & $(\langle \sigma \rangle \rtimes X) \times (P \rtimes \langle S A^{-1}, B^b \rangle)$ & $(C_p \rtimes C_d) \times (C_q \times (C_q \times C_r)) \rtimes C_b$\\
    \hline
    $b \in \{0,1\}$&$(\langle \sigma \rangle \rtimes X) \times (P \rtimes \langle S A^{\frac{1}{2}(r-1)}, B^b \rangle)$ & $(C_p \rtimes C_d) \times (\mathbb{F}_q^2 \rtimes_{\frac{1}{2}(r-1)} C_{br})$\\
    \hline
    $a \in \{0,1\}$, $1 \leq u \leq \frac{1}{2}(r-3)$ & $\langle \sigma, \alpha^a, \e_1,\e_2, SA^u, (\beta,B) \rangle$ & $((C_p \rtimes C_{q^a}) \times (\mathbb{F}_q^2 \rtimes_u C_r)) \rtimes C_2$\\
    \hline
    $a \in \{0,1\}$ & $\langle \sigma, \alpha^a, \e_1,\e_2, SA^{-1}, (\beta,B) \rangle$ & $((C_p \rtimes C_{q^a}) \times (C_q \times (C_q \times C_r))) \rtimes C_2$\\
    \hline
    $a \in \{0,1\}$ & $\langle \sigma, \alpha^a, \e_1,\e_2, SA^{\frac{1}{2}(r-1)}, (\beta,B) \rangle$ & $((C_p \rtimes C_{q^a}) \times (\mathbb{F}_q^2 \rtimes_{\frac{1}{2}(r-1)} C_r)) \rtimes C_2$\\
    \hline
    $b \in \{0,1\}$ & $(\langle \sigma \rangle \rtimes X) \times \langle \e_1,S,A,B^b \rangle$ & $(C_p \rtimes C_d) \times (C_q \rtimes (C_r \times C_{br}))$\\
    \hline
    $a \in \{0,1\}$ & $\langle \sigma,\alpha^a,\e_1,S,A,(\beta,B) \rangle$ & $((C_p \rtimes C_{q^a})\times (C_q \rtimes (C_r \times C_r))) \rtimes C_2$\\
    \hline
    $b \in \{0,1\}$ & $\langle \sigma,\beta^b,(\alpha,\e_1),SA^{-1} \rangle$ & $((C_p \rtimes C_b) \rtimes C_q) \times C_r$\\
    \hline
    $b \in \{0,1\}$ & $\langle \sigma, \beta^b, (\alpha,\e_1), \e_2, SA^{-1} \rangle$ & $((C_p \rtimes C_b) \rtimes C_q) \times (C_q \rtimes C_r)$\\
    \hline
    $b \in \{0,1\}$, $1 \leq u \leq r-2$ & $(\langle \sigma \rangle \rtimes X) \times \langle \e_1,SA^u,B^b \rangle$ & $(C_p \rtimes C_d) \times (C_q \rtimes C_{br})$\\
    \hline
    $b \in \{0,1\}$ & $(\langle \sigma \rangle \rtimes X) \times \langle \e_1,SA^{-1},B^b \rangle$ & $(C_p \rtimes C_d) \times (C_q \times C_{br})$\\
    \hline
    $a \in \{0,1\}$, $1 \leq u \leq r-2$ & $\langle \sigma,\alpha^a,\e_1,SA^u,(\beta,B) \rangle$ & $((C_p \rtimes C_{q^a}) \times (C_q \rtimes C_r)) \rtimes C_2$\\
    \hline
    $a \in \{0,1\}$ & $\langle \sigma,\alpha^a,\e_1,SA^{-1},(\beta,B) \rangle$ & $((C_p \rtimes C_{q^a}) \times (C_q \times C_r)) \rtimes C_2$\\
    \hline
    \end{tabular}}
    \vskip5mm
    \caption{Isomorphism classes for the transitive subgroups of $\Hol(N_3)$}
    \label{N3_isom_classes}
\end{sidewaystable}
\begin{sidewaystable}
    \setlength{\extrarowheight}{1.5mm}
    \vskip135mm
    \bigskip
    \centering
    \scalebox{1.25}{
        \begin{tabular}{|c|c|c|c|}
            \hline
            Structure & \# groups & $|\Aut(M,M')|$ & \# HGS\\
            \hline
            $(C_p \rtimes C_d) \times ((C_q \times (C_q \rtimes C_r)) \rtimes C_{br})$ & $1$ & $2q(p-1)(q-1)$ & $2$\\
            \hline
            $((C_p \rtimes C_{q^a}) \times ((C_q \times (C_q \rtimes C_r)) \rtimes C_r)) \rtimes C_2$ & $1$ & $2q(p-1)(q-1)$ & $2$\\
            \hline
            $(C_p \rtimes C_d) \times (\mathbb{F}_q^2 \rtimes_u C_{br})$ & $2$ & $q^{2-b}(p-1)(q-1)$ & $2q^{1-b}$\\
            \hline
            $(C_p \rtimes C_d) \times (C_q \times (C_q \times C_r)) \rtimes C_b$ & $2$ & $q(p-1)(q-1)$ & $2$\\
            \hline
            $(C_p \rtimes C_d) \times (\mathbb{F}_q^{2-b} \rtimes_{\frac{1}{2}(r-1)} C_{br})$ & $1$ & $2q^{2-b}(p-1)(q-1)$ & $2q^{1-b}$\\
            \hline
            $((C_p \rtimes C_{q^a}) \times (\mathbb{F}_q^2 \rtimes_u C_r)) \rtimes C_2$ & $2$ & $q^{2-b}(p-1)(q-1)$ & $2q^{1-b}$\\
            \hline
            $((C_p \rtimes C_{q^a}) \times (C_q \times (C_q \times C_r))) \rtimes C_2$ & $2$ & $q^{2-b}(p-1)(q-1)$ & $2q^{1-b}$\\
            \hline
            $((C_p \rtimes C_{q^a}) \times (\mathbb{F}_q^2 \rtimes_{\frac{1}{2}(r-1)} C_r))\rtimes C_2$ & $1$ & $2q^{2-b}(p-1)(q-1)$ & $2q^{1-b}$\\
            \hline
            $(C_p \rtimes C_d) \times (C_q \rtimes (C_r \times C_{br}))$ & $2q$ & $(p-1)(q-1)(r-1)$ & $2(r-1)$\\
            \hline
            $((C_p \rtimes C_{q^a})\times (C_q \rtimes (C_r \times C_r))) \rtimes C_2$ & $2q$ & $(p-1)(q-1)(r-1)$ & $2(r-1)$\\
            \hline
            $((C_p \rtimes C_b) \rtimes C_q) \times C_r$ & $2q(q-1)$ & $(p-1)(q-1)$ & $2(q-1)$\\
            \hline
            $((C_p \rtimes C_b) \rtimes C_q) \times (C_q \rtimes C_r)$ & $2(q-1)$ & $q(p-1)(q-1)$ & $2(q-1)$\\
            \hline
            $(C_p \rtimes C_d) \times (C_q \rtimes C_r)$ & $2q(r-2)+2$ & $q(p-1)(q-1)$ & $2q(r-2)+2$\\
            \hline
            $(C_p \rtimes C_d) \times (C_q \rtimes C_{2r})$ & $2q(r-1)$ & $(p-1)(q-1)$ & $2(r-1)$\\
            \hline
            $(C_p \rtimes C_d) \times (C_q \times C_{br})$ & $2q$ & $(p-1)(q-1)(r-1)$ & $2(r-1)$\\
            \hline
            $((C_p \rtimes C_{q^a}) \times (C_q \rtimes C_r)) \rtimes C_2$ & $2q(r-1)$ & $(p-1)(q-1)$ & $2(r-1)$\\
            \hline
            $((C_p \rtimes C_{q^a}) \times (C_q \times C_r)) \rtimes C_2$ & $2q$ & $(p-1)(q-1)(r-1)$ & $2(r-1)$\\
            \hline
        \end{tabular}
        }
        \vspace{5mm}
        \caption{Hopf--Galois structures of type $N_3$}
        \label{N3_HGS}
\end{sidewaystable}
We obtain the following theorem by counting the number of representatives in Table \ref{N3_isom_classes}.
\begin{theorem}\label{N3_tot}
    There are $\sigma_0(s)\left[\frac{1}{2}x(27+3q)+q+9\right]$ isomorphism types of permutation groups of degree $pqr$ which are realised by a Hopf--Galois structure of type $N_3$.
\end{theorem}

\newpage
\subsection{\texorpdfstring{$n=p_1p_2p_3p_4$}{n=p1p2p3p4}}
There are $F(4)=5$ groups of order $n=p_1p_2p_3p_4$ where $p_1=2p_2+1,p_2=2p_3+1,p_3=2p_4+1$. These are the cyclic group, $N_1:=C_n$ and the four metacyclic groups $N_2:=(C_{p_1}\rtimes C_{p_2}) \times C_{p_3p_4}$, $N_3:=(C_{p_2} \rtimes C_{p_3}) \times C_{p_1p_4}, N_4:=(C_{p_3} \rtimes C_{p_4}) \times C_{p_1p_2}, N_5:=(C_{p_1} \rtimes C_{p_2}) \times (C_{p_3} \rtimes C_{p_4})$.

We have $l=4$, and so our choices for $I$ are $\emptyset, \{2\}, \{3\}, \{4\}$, and $\{2,4\}$.

We look at the number of Hopf--Galois structures of type $N_1$. We have:
\begin{align*}
    &J_{\emptyset,\emptyset}=N_1,\\
    &J_{\{2\},t_2}=\left\langle \sigma_1, \left[\sigma_2,\alpha_1^{t_2}\right], \sigma_3,\sigma_4\right\rangle \cong N_2,\\
    &J_{\{3\},t_3}=\left\langle \sigma_1, \sigma_2, \left[\sigma_3,\alpha_2^{t_3}\right], \sigma_4 \right\rangle \cong N_3,\\
    &J_{\{4\},t_4}=\left\langle \sigma_1, \sigma_2, \sigma_3, \left[\sigma_4,\alpha_3^{t_4}\right] \right\rangle \cong N_4,\\
    &J_{\{2,4\},(t_2,t_4)}=\left\langle \sigma_1, \left[\sigma_2,\alpha_1^{t_2}\right], \sigma_3, \left[\sigma_4,\alpha_3^{t_4}\right] \right\rangle \cong N_5,
\end{align*}
where $1 \leq t_i \leq p_i-1$ for each $t_i$ above. Therefore, Proposition \ref{cyclic_trans} tells us that we have the following list of transitive subgroups of $\Hol(N_1)$:
\begin{align*}
    &N_1 \rtimes A               && A \leq \Aut(N_1),\\
    &J_{\{2\},t_2} \rtimes B   && B \leq \langle \beta_1,\alpha_3,\beta_3, \alpha_4,\beta_4 \rangle,\\
    &J_{\{3\},t_3} \rtimes C   && C \leq \langle \alpha_1,\beta_1,\beta_2, \alpha_4,\beta_4 \rangle,\\
    &J_{\{4\},t_4} \rtimes D   && D \leq \langle \alpha_1,\beta_1, \alpha_2, \beta_2, \beta_3, \rangle,\\
    &J_{\{2,4\},(t_2,t_4)} \rtimes E && E \leq \langle \beta_1,\beta_3 \rangle.
\end{align*}
Proposition \ref{cyclic_HGS_Cunningham} tells us that these correspond to extensions admitting either a unique Hopf--Galois structure, $p_2$, $p_3$, $p_4$, or $p_2p_4$ Hopf--Galois structures, depending on the choices of $B,C,D$ and $E$. Theorem \ref{cyclic_cunningham_thm} then tells us that there are
\[(452x+148)\sigma_0(s)+69\]
isomorphism types of permutation groups of degree $p_1p_2p_3p_4$ which are realised by a Hopf--Galois structure of type $N_1$. 

\section*{Acknowledgements}
This paper was completed under the support of the following two grants:

The Engineering and Physical Sciences Doctoral Training Partnership research grant EP/T518049/1 (EPSRC DTP).

Project OZR3762 of Vrije Universiteit
Brussel and FWO Senior Research Project G004124N.

\bibliography{MyBib}

\end{document}